\documentclass[10pt,oneside]{amsart}

\usepackage[
paper=a4paper,
headsep=20pt,text={136mm,208mm},centering,includehead
]{geometry}

\usepackage[bookmarks]{hyperref}
\usepackage{amssymb,amsxtra,dsfont}
\usepackage[all]{xy}
\usepackage{pb-diagram,pb-xy}
\usepackage{graphicx,color}
\usepackage{epstopdf}
\usepackage{pinlabel}
\usepackage{float}
\usepackage{array,calc,booktabs}
\usepackage[
  textwidth=1.2in,
  backgroundcolor=yellow,
  bordercolor=orange,
  textsize=small
]{todonotes}
\usepackage{tikz}

\iftrue
\makeatletter
\def\@maketitle{%
  \normalfont\normalsize
  \@adminfootnotes
  \@mkboth{\@nx\shortauthors}{\@nx\shorttitle}%
  \global\topskip42\p@\relax % 5.5pc   "   "   "     "     "
  \@settitle
  \ifx\@empty\authors \else \@setauthors \fi
  \ifx\@empty\@dedicatory
  \else
    \baselineskip22\p@
    \vtop{{\small\itshape\@dedicatory\@@par}%
      \global\dimen@i\prevdepth}\prevdepth\dimen@i
  \fi
  \@setabstract
  \normalsize
  \if@titlepage
    \newpage
  \else
    \dimen@25\p@ \advance\dimen@-\baselineskip
    \vskip\dimen@\relax
  \fi
} % end \@maketitle
\def\@settitle{%
  \vspace*{-15pt}
  \begin{flushleft}%
    \LARGE\bfseries
    \strut\@title\strut
  \end{flushleft}%
}
\def\@setauthors{%
  \begingroup
  \def\thanks{\protect\thanks@warning}%
  \trivlist
  \raggedright
  \large \@topsep27\p@\relax
  \advance\@topsep by -\baselineskip
  \item\relax
  \author@andify\authors
  \def\\{\protect\linebreak}%
  \authors
  \ifx\@empty\contribs
  \else
    ,\penalty-3 \space \@setcontribs
    \@closetoccontribs
  \fi
  \normalfont
  \endtrivlist
  \endgroup
}
\def\@setaddresses{\par
  \nobreak \begingroup
  \small\raggedright
  \def\author##1{\nobreak\addvspace\smallskipamount}%
  \def\\{\unskip, \ignorespaces}%
  \interlinepenalty\@M
  \def\address##1##2{\begingroup
    \par\addvspace\bigskipamount\noindent
    \@ifnotempty{##1}{(\ignorespaces##1\unskip) }%
    {\ignorespaces##2}\par\endgroup}%
  \def\curraddr##1##2{\begingroup
    \@ifnotempty{##2}{\nobreak\noindent\curraddrname
      \@ifnotempty{##1}{, \ignorespaces##1\unskip}\/:\space
      ##2\par}\endgroup}%
  \def\email##1##2{\begingroup
    \@ifnotempty{##2}{\nobreak\noindent E-mail address%
      \@ifnotempty{##1}{, \ignorespaces##1\unskip}\/:\space
      \ttfamily##2\par}\endgroup}%
  \def\urladdr##1##2{\begingroup
    \def~{\char`\~}%
    \@ifnotempty{##2}{\nobreak\noindent\urladdrname
      \@ifnotempty{##1}{, \ignorespaces##1\unskip}\/:\space
      \ttfamily##2\par}\endgroup}%
  \addresses
  \endgroup
  \global\let\addresses=\@empty
}
\def\@setabstracta{%
    \ifvoid\abstractbox
  \else
    \skip@15pt \advance\skip@-\lastskip
    \advance\skip@-\baselineskip \vskip\skip@
    \box\abstractbox
    \prevdepth\z@ % because \abstractbox is a vtop
    \vskip-15pt
  \fi
}
\renewenvironment{abstract}{%
  \ifx\maketitle\relax
    \ClassWarning{\@classname}{Abstract should precede
      \protect\maketitle\space in AMS document classes; reported}%
  \fi
  \global\setbox\abstractbox=\vtop \bgroup
    \normalfont\small
    \list{}{\labelwidth\z@
      \leftmargin0pc \rightmargin\leftmargin
      \listparindent\normalparindent \itemindent\z@
      \parsep\z@ \@plus\p@
      
    }%
    \item[\hskip\labelsep\bfseries\abstractname.]%
}{%
  \endlist\egroup
  \ifx\@setabstract\relax \@setabstracta \fi
}

\def\ps@headings{\ps@empty
  \def\@evenhead{%
    \setTrue{runhead}%
    \normalfont\scriptsize
    \rlap{\thepage}\hfill
    \def\thanks{\protect\thanks@warning}%
    \leftmark{}{}}%
  \def\@oddhead{%
    \setTrue{runhead}%
    \normalfont\scriptsize
    \def\thanks{\protect\thanks@warning}%
    \rightmark{}{}\hfill \llap{\thepage}}%
  \let\@mkboth\markboth
}\ps@headings

\def\section{\@startsection{section}{1}%
  \z@{-1.4\linespacing\@plus-.5\linespacing}{.8\linespacing}%
  {\normalfont\bfseries\Large}}
\def\subsection{\@startsection{subsection}{2}%
  \z@{-.8\linespacing\@plus-.3\linespacing}{.5\linespacing\@plus.2\linespacing}%
  {\normalfont\bfseries\large}}
\def\subsubsection{\@startsection{subsubsection}{3}%
  \z@{.7\linespacing\@plus.2\linespacing}{-1.5ex}%
  {\normalfont\bfseries}}
\def\@secnumfont{\bfseries}

\renewcommand\contentsnamefont{\bfseries}
\def\@starttoc#1#2{\begingroup
  \setTrue{#1}%
  \par\removelastskip\vskip\z@skip
  \@startsection{}\@M\z@{\linespacing\@plus\linespacing}%
    {.5\linespacing}{%\centering
      \contentsnamefont}{#2}%
  \ifx\contentsname#2%
  \else \addcontentsline{toc}{section}{#2}\fi
  \makeatletter
  \@input{\jobname.#1}%
  \if@filesw
    \@xp\newwrite\csname tf@#1\endcsname
    \immediate\@xp\openout\csname tf@#1\endcsname \jobname.#1\relax
  \fi
  \global\@nobreakfalse \endgroup
  \addvspace{32\p@\@plus14\p@}%
  \let\tableofcontents\relax
}
\def\contentsname{Contents}
\def\l@section{\@tocline{2}{.5ex}{0mm}{5pc}{}}
\def\l@subsection{\@tocline{2}{0pt}{2em}{5pc}{}}
\makeatother
\fi %\iftrue/false for amsart.cls hack

\def\to{\mathchoice{\longrightarrow}{\rightarrow}{\rightarrow}{\rightarrow}}
\makeatletter
\newcommand{\shortxra}[2][]{\ext@arrow 0359\rightarrowfill@{#1}{#2}}
\def\longrightarrowfill@{\arrowfill@\relbar\relbar\longrightarrow}
\newcommand{\longxra}[2][]{\ext@arrow 0359\longrightarrowfill@{#1}{#2}}
\renewcommand{\xrightarrow}[2][]{\mathchoice{\longxra[#1]{#2}}%
  {\shortxra[#1]{#2}}{\shortxra[#1]{#2}}{\shortxra[#1]{#2}}}

\def\addtagsub#1{\let\oldtf=\tagform@\def\tagform@##1{\oldtf{##1}\hbox{$_{#1}$}}}

\makeatother

\def\otimesover#1{\mathbin{\mathop{\otimes}_{#1}}}

\makeatletter
\def\Nopagebreak{\@nobreaktrue\nopagebreak}

\newtheoremstyle{theorem-giventitle}
        {}{}              %%% space between body and thm
        {\itshape}                      %%% Thm body font
        {}                              %%% Indent amount (empty = no indent)
        {\bfseries}                     %%% Thm head font
        {.}                             %%% Punctuation after thm head
        {\thm@headsep}                             %%% Space after thm head
        {\thmnote{\bfseries#3}}%%% Thm head spec
\newtheoremstyle{theorem-givenlabel}
        {}{}              %%% space between body and thm
        {\itshape}                      %%% Thm body font
        {}                              %%% Indent amount (empty = no indent)
        {\bfseries}                     %%% Thm head font
        {.}                             %%% Punctuation after thm head
        {\thm@headsep}                             %%% Space after thm head
        {\thmname{#1}~\thmnumber{#3}\setcurrentlabel{#3}}%%% Thm head spec

\newtheoremstyle{definition-giventitle}
        {}{}              %%% space between body and thm
        {}                      %%% Thm body font
        {}                              %%% Indent amount (empty = no indent)
        {\bfseries}                     %%% Thm head font
        {.}                             %%% Punctuation after thm head
        {\thm@headsep}                             %%% Space after thm head
        {\thmnote{\bfseries#3}}%%% Thm head spec
\def\setcurrentlabel#1{\gdef\@currentlabel{#1}}

\makeatother

\newtheorem{theorem}{Theorem}[section]

\newtheorem{proposition}[theorem]{Proposition}
\newtheorem{corollary}[theorem]{Corollary}
\newtheorem{lemma}[theorem]{Lemma}

\theoremstyle{definition}
\newtheorem{definition}[theorem]{Definition}

\newtheorem{remark}[theorem]{Remark}

\theoremstyle{theorem-giventitle}
\newtheorem{theorem-named}{}
\theoremstyle{theorem-givenlabel}
\newtheorem{theorem-labeled}{Theorem}

\theoremstyle{definition-giventitle}
\newtheorem{definition-named}{}
\newtheorem{step-named}{}

\numberwithin{equation}{section}

\def\Z{\mathbb{Z}}
\def\Q{\mathbb{Q}}
\def\R{\mathbb{R}}
\def\C{\mathbb{C}}
\def\NN{\mathbb{N}}
\def\KK{\mathbb{K}}

\def\N{\mathcal{N}}

\def\F{\mathcal{F}}
\def\G{\mathcal{G}}
\def\K{\mathcal{K}}

\def\cC{\mathcal{C}}

\def\cG{\mathcal{G}}
\def\cR{\mathcal{R}}
\def\cP{\mathcal{P}}

\def\Ker{\operatorname{Ker}}
\def\Coker{\operatorname{Coker}}
\def\Im{\operatorname{Im}}
\def\Hom{\operatorname{Hom}}

\def\sign{\operatorname{sign}}
\def\rank{\operatorname{rank}}

\def\ldim{\dim^{(2)}}
\def\lsign{\sign^{(2)}}

\def\Lt{L^2}
\def\rhot{\rho^{(2)}}

\def\setminus{\smallsetminus}

\def\fpn{\mathcal{F}^{cot,p}_n}
\def\fpnh{\mathcal{F}^{cot,p}_{n.5}}

\begin{document}

\vspace*{0mm}

\title[Filtrations of the knot concordance group]{Amenable signatures, algebraic solutions, and filtrations of the knot concordance group}

\author{Taehee Kim}
\address{
  Department of Mathematics\\
  Konkuk University \\
  Seoul 05029\\
  Korea
}
\email {tkim@konkuk.ac.kr}
\thanks{This research was supported by Basic Science Research Program through the National Research Foundation of Korea(NRF) funded by the Ministry of Education (no.~2011-0030044(SRC-GAIA) and no.~2015R1D1A1A01056634).}

\dedicatory{Dedicated to the memory of Tim D. Cochran}

\def\subjclassname{\textup{2010} Mathematics Subject Classification}
\expandafter\let\csname subjclassname@1991\endcsname=\subjclassname
\expandafter\let\csname subjclassname@2000\endcsname=\subjclassname
%\date{\today}

% \keywords{}

\begin{abstract}
It is known that each of the successive quotient groups of the grope and solvable filtrations of the knot concordance group has an infinite rank subgroup. The generating knots of these subgroups are constructed using iterated doubling operators. In this paper, for each of the successive quotients of the filtrations we give a new infinite rank subgroup which trivially intersects the previously known infinite rank subgroups. Instead of iterated doubling operators, the generating knots of these new subgroups are constructed using the notion of algebraic $n$-solutions, which was introduced by Cochran and Teichner. Moreover, for each slice knot $K$ whose Alexander polynomial has degree greater than 2, we construct the generating knots such that they have the same derived quotients and higher-order Alexander invariants up to a certain order.

In the proof, we use an $L^2$-theoretic obstruction for a knot to being $n.5$-solvable given by Cha, which is based on $L^2$-theoretic techniques developed by Cha and Orr. We also generalize  the notion of algebraic $n$-solutions to the notion of $R$-algebraic $n$-solutions where $R$ is either rationals or the field of $p$ elements for a prime $p$.

\end{abstract}

\maketitle

%\tableofcontents

\section{Introduction}

In this paper, we address the structure of the grope and solvable filtrations of the knot concordance group. Two oriented knots $K_0$ and $K_1$ in the 3-sphere $S^3$ are {\it concordant} if there is a smoothly and properly embedded annulus in $S^3\times [0,1]$ whose boundary is the union of $K_0\times\{0\}$ and $-K_1\times \{1\}$. It is known that $K_0$ and $K_1$ are concordant if and only if the connected sum $K_0\# (-K_1)$ bounds a smoothly embedded disk in the 4-ball $D^4$, namely, $K_0\# (-K_1)$ is a {\it slice knot}. 
Concordance classes form an abelian group under connected sum, which is called {\it the knot concordance group}, and we denote it by $\cC$.

The notion of concordance on knots was introduced by Fox and Milnor in the 1950's. In the 1960's Levine classified the algebraic concordance group \cite{Levine:1969-1,Levine:1969-2}, and in the 1970's Casson and Gordon showed that the surjection from $\cC$ to the algebraic concordance group is not injective \cite{Casson-Gordon:1986-1}.

In the late 1990's, Cochran, Orr, and Teichner \cite{Cochran-Orr-Teichner:1999-1} introduced the grope and solvable filtrations of the knot concordance group denoted by $\{\G_n\}$ and $\{\F_n\}$, respectively, which are indexed by nonnegative half-integers. The subgroup $\G_n$ consists of knots bounding a {\it grope} of height $n$ in $D^4$, where a grope is a certain 2-complex constructed by attaching surfaces along their boundaries. See Definition~\ref{def:grope} for a precise definition of a grope. Similarly, $\F_n$ is the subgroup of knots such that the zero-framed surgery on the knot bounds an {\it $n$-solution}, where an $n$-solution is a 4-manifold satisfying certain conditions on the (equivariant) intersection form with twisted coefficients (see Definition~\ref{def:n-solution}). For all $n$ it is known that $\G_{n+2}\subset \F_n$ \cite[Theorem~8.11]{Cochran-Orr-Teichner:1999-1}, and one may consider an $n$-solution as an order $n$ approximation of the exterior of a slice disk in $D^4$.  These filtrations reflect classical invariants at low levels. For instance, a knot $K$ has vanishing Arf invariant if and only if $K\in \F_0$, and $K$ is algebraically slice if and only if $K\in \F_{0.5}$ \cite{Cochran-Orr-Teichner:1999-1}. Furthermore, a knot in $\F_{1.5}$ has vanishing Casson--Gordon invariants \cite{Cochran-Orr-Teichner:1999-1}, but it is known that there exists a knot with vanishing Cassson--Gordon invariants which is not in $\F_{1.5}$ \cite{Kim:2004-1}. 

In this paper, for each of the successive quotients of the grope and solvable filtrations we give a new infinite rank subgroup which trivially intersects the previously known infinite rank subgroups.

We recall results on finding infinite rank subgroups of the successive quotients $\G_{n+2}/\G_{n+2.5}$ and $\F_n/\F_{n.5}$ for $n\ge 2$. We are interested in the cases for $n\ge 2$ since the cases for $n\le 2$ were well-known from classical invariants. For finding infinite rank subgroups of the quotients, there are three different approaches using the following: 1) rationally universal solvable representations; 2) algebraic $n$-solutions; 3) iterated doubling operators. In this paper, by an iterated doubling operator we mean any of iterated (generalized) doubling operators used in \cite{Cochran-Harvey-Leidy:2009-1, Cochran-Harvey-Leidy:2009-2,Cha:2010-1,Horn:2010-1}. For a more precise definition of an iterated doubling operator, see the proof of Theorem~\ref{thm:refined_main_theorem-2}.

First, in their seminal papers \cite{Cochran-Orr-Teichner:1999-1,Cochran-Orr-Teichner:2002-1}, using the von Neumann--Cheeger--Gromov $\rhot$-invariants associated to {\it rationally universal solvable representations}, Cochran, Orr, and Teichner showed that $\F_2/\F_{2.5}$ has an infinite rank subgroup, giving the first example of nonslice knots with vanishing Casson--Gordon invariants. For each integer $n\ge 2$, Cochran and Teichner showed that $\F_n/\F_{n.5}$ and $\G_{n+2}/\G_{n+2.5}$ are infinite and have positive rank using Cheeger--Gromov's universal bound on $\rhot$-invariants and their new notion of an {\it algebraic $n$-solution} \cite{Cochran-Teichner:2003-1}. Later Cochran--Teichner's work was refined further by Cochran and the author \cite{Cochran-Kim:2004-1}. 
Also, in \cite{Cochran-Kim:2004-1} the notion of an algebraic $n$-solution was generalized.

Instead of algebraic $n$-solutions, using {\it iterated doubling operators} and unlocalized higher-order Blanchfield linking forms Cochran, Harvey, and Leidy obtained the first example of an infinite rank subgroup of $\F_n/\F_{n.5}$ for each integer $n\ge 2$ \cite{Cochran-Harvey-Leidy:2009-1}. They extended their work further and showed that the solvable filtration $\{\F_n\}$ has refined filtrations related to primary decomposition whose successive quotient groups have infinite rank \cite{Cochran-Harvey-Leidy:2009-2}.   For the grope filtration, Horn \cite{Horn:2010-1} proved that $\G_{n+2}/\G_{n+2.5}$ has an infinite rank subgroup for each integer $n\ge 2$, where the generating knots of the subgroups were constructed using iterated doubling operators. 

In all of the above work, knots were obstructed to being in $\F_{n.5}$ (hence not in $\G_{n+2.5}$) using the $\rhot$-invariants associated to a representation mapping to a poly-torsion-free-abelian (henceforth PTFA) group. By a PTFA group, we mean a group which allows a subnormal series whose successive quotient groups are torsion-free and abelian. These obstructions are essentially based on the vanishing criterion of the $\rhot$-invariants associated to a PTFA representation  given by Cochran--Orr--Teichner \cite[Theorem 4.2]{Cochran-Orr-Teichner:1999-1}. For the reader's convenience, we review the $\rhot$-invariant in Subsection~\ref{subsec:amenable signature}, where the $\rhot$-invariant is defined to be an $L^2$-signature defect. 

In \cite{Cha-Orr:2009-1}, Cha and Orr developed $L^2$-theoretic methods: for a representation mapping to a group which is amenable and lies in Strebel's class $D(R)$ for some commutative ring $R$, they proved the homology cobordism invariance of the $L^2$-Betti numbers and the $\rhot$-invariants, and presented a method for controlling the $L^2$-dimension of homology with $L^2$-coefficients. (The reader may refer to \cite{Cha-Orr:2009-1} for definitions of amenable groups and Strebel's class $D(R)$, but the definitions will not be needed in this paper.) 

Based on the work in \cite{Cha-Orr:2009-1}, Cha found a vanishing criterion of $\rhot$-invariants for $n.5$-solvable knots in \cite[Theorem~4.2]{Cochran-Orr-Teichner:1999-1} to include the vanishing of the $\rhot$-invariants associated to a representation mapping to a group which is amenable and lies in Strebel's class $D(R)$ where $R=\Q$ or $\Z_p$ for a prime $p$ \cite[Therorem 1.3]{Cha:2010-1}. In this paper,  we call this extended vanishing criterion {\it Amenable Signature Theorem (for $n.5$-solvability)} (see Theorem~\ref{thm:obstruction}). We note that  the above class of groups includes PTFA groups and some groups with torsion elements (see Lemma~\ref{lem:amenable and D(R)}). Using Amenable Signature Theorem~\ref{thm:obstruction}, Cha constructed an infinite rank subgroup of $\F_n/\F_{n.5}$ for each integer $n\ge 2$ for which the $\rhot$-invariants associated to a PTFA representation vanish \cite[Theorem 1.4]{Cha:2010-1}. The generating knots of the infinite rank subgroups in \cite{Cha:2010-1} were constructed also using iterated doubling operators. 

In this paper, we combine algebraic $n$-solutions and Amenable Signature Theorem to find infinite rank subgroups of the grope and solvable filtrations, and give Theorem~\ref{thm:main} below. In the following, $\Delta_K(t)$ denotes the Alexander polynomial of a knot $K$. For a group $G$, let $G^{(0)}:=G$, and inductively for $n\ge 0$ define $G^{(n+1)}:=[G^{(n)},G^{(n)}]$. The group $G^{(n)}$ is called the {\it $n$-th derived group} of $G$. Also recall that $\cG_{n+2}\subset \F_n$ for all~$n$. 

\begin{theorem}\label{thm:main} Let $n\ge 2$ be an integer. Let $K$ be a slice knot with $\deg \Delta_K(t) > 2$. Then there is a sequence of knots $K_1,K_2,\ldots$  satisfying the following.
\begin{enumerate}
	\item For each $i$, there is an isomorphism
	\[
	\pi_1(S^3\setminus K_i)/\pi_1(S^3\setminus K_i)^{(n+1)} \to \pi_1(S^3\setminus K)/\pi_1(S^3\setminus K)^{(n+1)}
	\]
	which preserves the peripheral structures.
	\item The $K_i$ are in $\cG_{n+2}$ and they are linearly independent modulo $\F_{n.5}$. In particular, the knots $K_i$ generate an infinite rank subgroup in each of $\cG_{n+2}/\cG_{n+2.5}$ and $\F_n/\F_{n.5}$.
	\item In $\cG_{n+2}/\cG_{n+2.5}$, the infinite rank subgroup generated by the $K_i$ trivially intersects the infinite rank subgroup in \cite{Horn:2010-1}.
	\item In $\F_n/\F_{n.5}$, the infinite rank subgroup generated by the $K_i$ trivially intersects the infinite rank subgroups in \cite{Cochran-Harvey-Leidy:2009-1, Cochran-Harvey-Leidy:2009-2,Cha:2010-1}.
\end{enumerate}

\end{theorem}
In Theorem~\ref{thm:main}, the condition $\deg \Delta_K(t)>2$ is best possible; by the work of Friedl and Teichner \cite{Friedl-Teichner:2005-1}; Theorem~\ref{thm:main} does not hold for any $n\ge 2$ for a certain slice knot $K$ with $\deg \Delta_K(t)=2$ (also refer to \cite[Proposition~5.10]{Cochran-Kim:2004-1}). Theorem~\ref{thm:main} is an immediate consequence of Theorems~\ref{thm:refined main theorem-1} and \ref{thm:refined_main_theorem-2}. 
As in Theorems~\ref{thm:main}(3) and (4), the infinite rank subgroups generated by the $K_i$ trivially intersect the previously known infinite rank subgroups, which were constructed using iterated doubling operators. 

An advantage of constructing the $K_i$ using not iterated doubling operators but algebraic $n$-solutions is that Theorem~\ref{thm:main}(1) is obtained for \emph{any} slice knot with $\deg \Delta_K(t)> 2$. Theorem~\ref{thm:main}(1) is related to studying the structure on knot concordance under a fixed Seifert form. It is well-known by the work of Freedman \cite{Freedman:1982-1,Freedman-Quinn:1990-1} that a knot with trivial Alexander polynomial is topologically slice, and hence determines a unique topological concordance class, namely, the class of topologically slice knots. It was asked if there is any other Alexander polynomial or a Seifert form which determines a unique topological concordance class. This question was answered negative that for each Seifert form of a knot $K$ with nontrivial Alexander polynomial, there exists infinitely many mutually topologically nonconcordant knots $K_i$ having the Seifert form \cite{Livingston:2002-1,Kim:2005-1}. (Also refer to \cite{Kim:2017-1}.) This result was refined further that in \cite[Theorem~5.1]{Cochran-Kim:2004-1} it was shown under the condition $\deg \Delta_K(t)>2$ that for each $n\ge 2$ those mutually nonconcordant knots $K_i$ can be constructed using algebraic $n$-solutions such that the $K_i$ are mutually distinct in $\F_n/\F_{n.5}$ and satisfies the property in Theorem~\ref{thm:main}(1). Furthermore, in \cite{Cochran-Kim:2004-1} it was shown that the $K_i$ have the same \emph{$m$th order Seifert presentation} (see \cite[Definitions~5.5 and 5.9]{Cochran-Kim:2004-1}) as $K$ for $m=0,1,\ldots, n-1$, and hence have the same Seifert form. In this paper, in Theorem~\ref{thm:main}, we refine it further and construct the $K_i$ such that they are linearly independent in $\F_n/\F_{n.5}$, still satisfying the condition in Theorem~\ref{thm:main}(1). We also note that the $K_i$ in Theorem~\ref{thm:main} also have the same $m$th order Seifert presentation for $m=0,1,\ldots, n-1$. (This can be easily shown using the same arguments in the proof of \cite[Theorem~5.1]{Cochran-Kim:2004-1} and will not be discussed in this paper.)

We also refer the reader to \cite{Kim:2006-1} and \cite{Cha-Kim:2008-1} for more applications of algebraic $n$-solutions to doubly slice knots and the solvable filtration of the \emph{rational} knot concordance group.

We construct the $K_i$ in Theorem~\ref{thm:main} using a well-known process called {\it infection} or {\it satellite construction} which involve a {\it seed knot}, {\it axes (knots)}, and {\it infection knots} (see Section~\ref{sec:construction}). Then, we show that the $K_i$ are linearly independent modulo $\F_{n.5}$ using Amenable Signature Theorem~\ref{thm:obstruction} following the ideas in \cite{Cha:2010-1}. We give more details: let $J$ be a nontrivial linear combination of the $K_i$. For an $n$-solution $V$ for~$J$, as obstructions for $V$ to being an $n.5$-solution, we use the $\rhot$-invariants associated to a representation factoring through $\pi_1W\to \pi_1W/\cP^{n+1}\pi_1W$ where $W$ is a certain 4-manifold containing $V$. Here, $\cP^{n+1}\pi_1W$ is a subgroup of $\pi_1W$ obtained from a certain {\it mixed-coefficient commutator series} of $\pi_1W$ which is defined depending on the choice of a prime $p$. See Definition~\ref{def:commutator series} for a precise definition of $\cP^{n+1}\pi_1W$. We remark that as can be seen in the proof of Theorem~\ref{thm:refined main theorem-1} the choice of a prime $p$ is specific to the linear combination $J$. Namely, we can choose a representation specific to $J$ and use the corresponding $\rhot$-invariant, and this makes easier to show that $J$ is not trivial modulo $\F_{n.5}$. In the representation, the quotient group $\pi_1W/\cP^{n+1}\pi_1W$ has torsion elements, hence is not a PTFA group. But it is amenable and lies in Strebel's class $D(\Z_p)$, and this fact enables us to use Amenable Signature Theorem~\ref{thm:obstruction} and obtain desired computations of $\rhot$-invariants.

When we construct the $K_i$ using infection, we need to choose the axes in a subtle way: for a given slice knot $K$ with $\deg \Delta_K(t) >2$, we use $K$ as a seed knot and choose axes $\eta_i$, $1\le i \le m$, such that for each homomorphism on $\pi_1$ induced from the inclusion $M(K)\to W$ where $M(K)$ is the 0-surgery on $K$ and $W$ is an $n$-solution for $K$, the axes $\eta_i$ should satisfy a certain nontriviality property under the homomorphism. The desired nontriviality property and existence of such axes are given in Theorem~\ref{thm:eta_i}, which is a key technical theorem in this paper. Roughly speaking, the nontriviality  property requires that for a given $n$-solution $W$ and $\cP^{n+1}\pi_1W $, there exists some $\eta_i$ such that $\eta_i\notin \cP^{n+1}\pi_1W $. To prove Theorem~\ref{thm:eta_i} we generalize the notion of an algebraic $n$-solution in \cite{Cochran-Teichner:2003-1,Cochran-Kim:2004-1} and define the notion of an {\it $R$-algebraic $n$-solution} where $R=\Q$ or $\Z_p$ for a prime $p$ (see Definition~\ref{def:algebraic n-solution}). We also use a nontriviality theorem on homology with twisted coefficients which is obtained using higher-order Blanchfield linking forms (see Theorem~\ref{thm:nontrivial} and the proof of Theorem~\ref{thm:eta_i}).

We prove Theorems~\ref{thm:main}(3) and (4) in Theorem~\ref{thm:refined_main_theorem-2} following the ideas in \cite[Section 9]{Cochran-Harvey-Leidy:2009-2}. For each positive half-integer $n$, a prime $p$, and a group $G$, we define a subgroup $G^{(n)}_{cot,p}$ of $G$ using a localization of a group ring (see Subsection~\ref{subsec:distinction}). Then, using it we define the notion of \emph{$(n,p)$-solvable knots} (Definition~\ref{def:p-F_n}). We note that $G^{(n)}\subset G^{(n)}_{cot,p}$ for all prime $p$, and an $n$-solvable knot is $(n,p)$-solvable for all prime $p$. Finally, we show that a nontrivial linear combination of the $K_i$ in Theorem~\ref{thm:main} is not $(n.5,p)$-solvable for some prime $p$, but a knot concordant to a linear combination of the knots in \cite{Horn:2010-1,Cochran-Harvey-Leidy:2009-1, Cochran-Harvey-Leidy:2009-2,Cha:2010-1} is $(n.5,p)$-solvable for all prime $p$. 

Other than finding infinite rank subgroups of $\F_n/\F_{n.5}$ and $\G_{n+2}/\G_{n+2.5}$, there are many interesting results on the grope and solvable filtrations \cite{Kim-Kim:2008-1,Cochran-Harvey-Leidy:2009-2,Cochran-Harvey-Leidy:2009-3,Burke:2014-1,Davis:2014-1,Kim-Kim:2014-1,Jang:2015-1}. For instance, it is known that for each integer $n\ge 2$ and for each of $\F_n/\F_{n.5}$ and $\cG_{n+2}/\cG_{n+2.5}$ there exists a subgroup infinitely generated by knots of order 2 \cite{Cochran-Harvey-Leidy:2009-3,Jang:2015-1}. These knots are constructed using iterated doubling operators, and it is unknown whether or not one can construct a subgroup infinitely generated by knots of order 2 using algebraic $n$-solutions.

We can also define the grope and solvable filtrations $\{\G_n^{top}\}$ and $\{F_n^{top}\}$ of the topological knot concordance group \cite{Cochran-Orr-Teichner:1999-1}. We note that all the results in this paper also hold under this topological setting; in this paper all the examples of manifolds and gropes are constructed smoothly, and the obstructions obtained from Amenable Signature Theorem~\ref{thm:obstruction} are topological obstructions. Moreover, by the work of Freedman and Quinn \cite{Freedman-Quinn:1990-1}, it is known that in fact the smooth and topological filtrations are equivalent, that is,  $K\in \G_n$ (resp. $K\in \F_n$) if and only $K\in \G_n^{top}$ (resp. $K\in \F_n^{top}$), see \cite[Remark~2.19]{Cha:2012-1} and \cite[p.454]{Cochran-Harvey-Leidy:2009-2}.

This paper is organized as follows. In Section~\ref{sec:preliminaries}, we introduce the grope and solvable filtrations of the knot concordance group and give Amenable Signature Theorem~\ref{thm:obstruction}. 
In Section~\ref{sec:construction}, we discuss the construction of examples using infection (or satellite construction). We prove Theorem~\ref{thm:main} in Section~\ref{sec:infinite rank subgroup}, and discuss higher-order Blanchfield linking forms and the notion of $R$-algebraic $n$-solutions in Section~\ref{sec:Blanchfield linking form and algebraic n-solutions}.

In this paper, manifolds are assumed to be smooth, compact, and oriented, and $\Z_p$ denotes the field of $p$ elements for a prime $p$. By abuse of notation we use the same symbol for a knot and its homotopy class and homology classes. Homology groups come with integer coefficients unless specified otherwise. For a knot 
$K$, we denote by $E(K)$ and $M(K)$ the exterior of $K$ in $S^3$ and the 0-surgery on $K$ in $S^3$, respectively.

%\subsection*{Acknowledgments} 

\section{Preliminaries}\label{sec:preliminaries}
In this section, we review the grope and solvable filtrations of the knot concordance group, and recall necessary results on amenable signatures and mixed-coefficient commutator series of a group in \cite{Cha:2010-1}. 

\subsection{The grope and solvable filtrations}\label{subsec:filtration}
In this subsection we review the notions of a grope, an $n$-solution, an $n$-cylinder and the grope and solvable filtrations of the knot concordance group.

\begin{definition}\cite{Freedman-Teichner:1995-1}\cite[Definition 7.9]{Cochran-Orr-Teichner:1999-1}\label{def:grope}
	A {\it grope of height 1} is a compact connected surface with a single boundary component. This boundary component is called {\it the base circle}. Let $\Sigma$ be a grope of height 1 of genus $g$. Let $\{\alpha_i, \beta_i\}_{1\le i\le g}$ be a standard symplectic basis of circles on $\Sigma$ such that $\alpha_i$ and $\beta_i$ are dual to each other. For an integer $n\ge 1$, a {\it grope of height $n+1$} is a 2-complex obtained by attaching gropes of height $n$ to each $\alpha_i$ and $\beta_i$ along the base circles. A {\it grope of height $n.5$} is a 2-complex obtained by attaching gropes of height $n$ to each $\alpha_i$ and gropes of height $n-1$ to each $\beta_i$ along the base circles. Here, the surface $\Sigma$ is called {\it the bottom stage} of the grope, and a grope of height 0 is understood to be the empty set. For a grope embedded in a 4-manifold, we require the grope has a neighborhood which is diffeomorphic to the product of $\R$ and the standard neighborhood of the (abstract) grope in $\R^3$. In this paper, a grope in a 4-manifold means a grope smoothly embedded in the 4-manifold in this way. We also note that in the literature a grope defined in this way is called a {\it symmetric grope}.
\end{definition}

Let $\NN_0 = \NN\cup \{0\}$. For each $n\in \frac12 \NN_0$, we denote by $\cG_n$ the subset of $\cC$ which consists of knots bounding a grope of height $n$ in $D^4$. It is known that each $\G_n$ is a subgroup of $\cC$ and $\G_m\subset \G_n$ for $m\ge n$, and hence $\{\G_n\}$ is a filtration of $\cC$. We call it the {\it grope filtration} of $\cC$.

For a group $G$, recall that $G^{(n)}$ denotes the $n$-th derived group of $G$. 
Now for a 4-manifold $W$, let $\pi:=\pi_1W$ and let $R:=\Z$, $\Q$, or $\Z_p$. Then for each $n\ge 0$,  there is the equivariant intersection form 
\[
\lambda_n^R\colon H_2(W;R[\pi/\pi^{(n)}]) \times H_2(W;R[\pi/\pi^{(n)}]) \to R[\pi/\pi^{(n)}] .
\] 
We drop the decoration $R$ from $\lambda_n^R$ when it is understood from the context. Below, we generalize the notions of an $n$-cylinder and a rational $n$-cylinder in \cite{Cochran-Kim:2004-1}.

\begin{definition}(\cite[Section 2]{Cochran-Kim:2004-1} for $R=\Z$ and $\Q$)
\label{def:n-cylinder}
	Let $R=\Z$, $\Q$, or $\Z_p$. Let $n$ be a nonnegative integer.  Let $W$ be a compact connected 4-manifold with $\partial W=\coprod_{i=1}^\ell M_i$ where each $M_i$ is a connected component with $H_1(M_i)\cong R$. Let $\pi:=\pi_1W$ and $r=\frac12 \rank_R\Coker\{H_2(\partial W;R)\to H_2(W;R)\}$.
	\begin{enumerate}
		\item $W$ is an {\it $R$-coefficient $n$-cylinder} if each inclusion from $M_i$ to $W$ induces an isomorphism on $H_1(M_i;R)$ and there exist $x_1, x_2, \ldots, x_r$ and $y_1, y_2, \ldots, y_r$ in $H_2(W;R[\pi/\pi^{(n)}])$ such that $\lambda_n(x_i,x_j)=0$ and $\lambda_n(x_i,y_j)=\delta_{ij}$ for $1\le i,j\le r$.  
		\item $W$ is an {\it $R$-coefficient $n.5$-cylinder} if $W$ satisfies (1), and furthermore there exist lifts $\tilde{x}_1,\tilde{x}_2,\ldots, \tilde{x}_r$ of $x_1,x_2,\ldots, x_r$ in $H_2(W;R[\pi/\pi^{(n+1)}])$ such that $\lambda_{n+1}(\tilde{x}_i, \tilde{x}_j)=0$ for $1\le i,j\le r$.
	\end{enumerate}
	We also require $W$ to be spin when $R=\Z$. A $\Z$-coefficient $n$-cylinder is also called an {\it $n$-cylinder}. The submodule generated by $x_1, x_2,\ldots, x_r$ (resp. $\tilde{x}_1,\tilde{x}_2,\ldots, \tilde{x}_r$) is called an {\it $n$-Lagrangian} (resp. {\it $(n+1)$-Lagrangian}), and the submodule generated by $y_1,y_2, \ldots, y_r$ is called its {\it $n$-dual}. 
\end{definition}

It is obvious that an $n$-cylinder is a $\Q$-coefficient $n$-cylinder. We also have the following proposition.

\begin{proposition}\label{prop:implication of cylinder}
	An $n$-cylinder is a $\Z_p$-coefficient $n$-cylinder. 
\end{proposition}
\begin{proof}
	Let $W$ be an $n$-cylinder with connected boundary components $M_i$, $1\le i\le \ell$. Since $H_1(M_i)\cong H_1(W)\cong \Z$, the map $H_1(M_i;\Z_p)\to H_1(W;\Z_p)$ is an isomorphism for each $i$. For $R=\Z$ or $\Z_p$, let $r(R)=\rank_R\Coker\{H_2(\partial W;R)\to H_2(W;R)\}$. By naturality  of intersection forms, it suffices to show that $r(\Z)=r(\Z_p)$.  From the long exact sequence
	\begin{multline*}
	0\to \Coker\{H_2(\partial W;R)\to H_2(W;R)\} \\
	\to H_2(W,\partial W;R)\to H_1(\partial W;R)\to H_1(W;R)\to 0
	\end{multline*}
	and the fact that $\rank_R H_2(W,\partial W;R) = \rank_R H_2(W;R)$, we obtain that \[
r(R) = \rank_R H_2(W;R) -(\ell -1).
\] 
Since $H_1(M_i)\cong H_1(W)\cong \Z$, the groups $H_1(W)$ and $\Coker\{H_1(\partial W)\to H_1(W)\}$ have no $p$-torsion. Therefore, by \cite[Lemma~3.14]{Cha:2010-1} we have $\rank_\Z H_2(W)=\rank_{\Z_p} H_2(W;\Z_p)$. Therefore $r(\Z)=r(\Z_p)$.	
\end{proof}

An $n$-cylinder in \cite{Cochran-Kim:2004-1} appeared as a generalization of an $n$-solution in \cite{Cochran-Orr-Teichner:1999-1}. Below we generalize the notion of an $n$-solution to the notion of an $R$-coefficient $n$-solution. 

\begin{definition}(\cite[Section 8]{Cochran-Orr-Teichner:1999-1} for $R=\Z$ and $\Q$)
\label{def:n-solution}
	An $R$-coefficient $n$-cylinder with a single boundary component is also called an {\it $R$-coefficient $n$-solution}. A $\Z$-coefficient $n$-solution is called an {\it $n$-solution}. A closed 3-manifold $M$ with $H_1(M)\cong \Z$ is {\it $n$-solvable via $W$} if there exists an $n$-solution $W$ with boundary $M$. A knot $K$ in $S^3$ is {\it $n$-solvable via $W$} and $W$ is an {\it $n$-solution for $K$} if $M(K)$ is $n$-solvable via $W$.
\end{definition}

For each $n\in \frac12 \NN_0$, we denote by $\F_n$ the subset of $n$-solvable knots in $\cC$. It is known that $\F_n$ is a subgroup of $\cC$ and $\F_m\subset \F_n$ if $m\ge n$ \cite{Cochran-Orr-Teichner:1999-1}. Therefore, $\{\F_n\}$ is a filtration of $\cC$, and we call it the {\it solvable filtration} of $\cC$. The filtrations $\{\cG_n\}$ and $\{\F_n\}$ have the following relationship.

\begin{theorem}\label{thm:G_n and F_n}\cite[Theorem 8.11]{Cochran-Orr-Teichner:1999-1}
	Let $n\in \frac12 \NN_0$. If a knot $K$ bounds a grope of height $n+2$ in $D^4$, then $K$ is $n$-solvable. That is, $\G_{n+2}\subset \F_n$. 
\end{theorem}
It is unknown whether or not the converse $\F_n\subset \G_{n+2}$ holds in general, but it is known that $\F_n=\G_{n+2}$ when $n=0,\,0.5$ \cite[Theorem~8.13 and Remark~8.14]{Cochran-Orr-Teichner:1999-1}.

\subsection{Amenable signatures}\label{subsec:amenable signature}

In this subsection, we briefly review the von Neumann--Cheeger--Gromov $\rhot$-invariants \cite{Cheeger-Gromov:1985-1}, and introduce Amenable Signature Theorem~\ref{thm:obstruction}. For a closed 3-manifold $M$, let $\Gamma$ be a countable  group and $\psi\colon \pi_1M\to \Gamma$ a group homomorphism. Then Chang and Weinberger \cite{Chang-Weinberger:2003-1} showed that there exists a group $G$ containing $\Gamma$ and a 4-manifold $W$ such that  $\phi\colon \pi_1M\xrightarrow{\psi} \Gamma\hookrightarrow  G$ extends to $ \pi_1W\to G$. Then, for $\N G$, which is the group von Neumann algebra of $G$, using $\phi$ we obtain a homomorphism $\Z[\pi_1W]\to \N G$ and the corresponding equivariant hermitian intersection form
\[
\lambda\colon H_2(W;\N G)\times H_2(W;\N G)\to \N G
\]
and the $\Lt$-signature $\lsign_G(W) \in \R$. Let $\sign(W)$ be the ordinary signature of $W$ and let $S_G(W):=\lsign_G(W) - \sign(W)$, an $\Lt$-signature defect. Then the {\it von Neumann--Cheeger--Gromov $\rhot$-invariant associated to $(M,\psi)$} is defined to be
\[
\rhot(M,\psi) := S_G(W).
\]
It is known that it is independent of the choices of a group $G$ and a 4-manifold of $W$. For more details on the von Neumann--Cheeger--Gromov $\rhot$-invariants, refer to \cite{Cochran-Teichner:2003-1, Cha:2010-1}.

The following lemma gives the existence of a universal upper bound on the $\rhot$-invariants for a fixed closed 3-manifold. 
\begin{lemma}\label{lem:universal bound}
	\begin{enumerate}
		\item \cite{Cheeger-Gromov:1985-1,Ramachandran:1993-1} For a closed 3-manifold $M$, there exists a constant $C_M$ such that $|\rhot(M,\phi)|\le C_M$ for every homomorphism $\phi\colon \pi_1M\to G$ where $G$ is a group. 
		\item \cite{Cha:2014-1} Furthermore, if $M=M(K)$ for a knot $K$ with crossing number $c(K)$, then we can take $C_M=69713280\cdot c(K)$.
	\end{enumerate}
\end{lemma}

Using $\Lt$-theoretic techniques and the amenable signature theorem on homology bordism developed in \cite{Cha-Orr:2009-1}, Cha obtained the following obstruction for a knot to being $n.5$-solvable, which generalizes the obstruction from the $\rhot$-invariant associated with a PTFA representation in \cite{Cochran-Orr-Teichner:1999-1}.

\begin{theorem}[Amenable Signature Theorem for $n.5$-solvability]
\label{thm:obstruction}
\cite[Theorem 1.3]{Cha:2010-1}
	Let $K$ be an $n.5$-solvable knot. Let $G$ be an amenable group lying in Strebel's class $D(R)$ for $R=\Q$ or $\Z_p$ such that $G^{(n+1)}=\{e\}$. Let $\phi\colon \pi_1M(K)\to G$ be a homomorphism which sends the meridian of $K$ to an infinite order element in $G$. Suppose $\phi$ extends to an $n.5$-solution for $M(K)$. Then, $\rhot(M(K),\phi)=0$.
\end{theorem}
The only amenable groups in Strebel's class $D(R)$ which we will use in this paper are the groups given in Lemma~\ref{lem:amenable and D(R)} below, and we will not need the definitions of amenable group and Strebel's class $D(R)$. (One may find the definitions in \cite{Cha-Orr:2009-1}.)

The computation of $\rhot$-invariants is not easy in general, but for the $\rhot$-invariants associated to an abelian representation, we have the lemma below. For a knot $K$, let $\sigma_K$ be the Levine-Tristram signature function. That is, for $\omega\in S^1\subset \C$, $\sigma_K(\omega)$ is the signature of the hermitian matrix $(1-\omega)A+(1-\bar{\omega})A^T$ where $A$ is a Seifert matrix for $K$.

\begin{lemma}\label{lem:computation of rho-invariant}\cite[Proposition 5.1]{Cochran-Orr-Teichner:2002-1}\cite[Corollary 4.3]{Friedl:2003-5}
	Let $K$ be a knot and let $\mu$ be the meridian of $K$. If $\phi\colon \pi_1M(K)\to G$ is a homomorphism whose image is abelian, then
	\[
	\rhot(M(K),\phi)= 
	\begin{cases}
	\int_{S^1} \sigma_K(\omega)\,d\omega & \mbox{if } \phi(\mu)\in G \mbox{ has infinite order} \\[1ex]
	\sum_{r=0}^{d-1}\sigma_K(e^{2\pi r\sqrt{-1}/d}) & \mbox{if } \phi(\mu)\in G \mbox{ has finite order } d.
	\end{cases}
	\]
	
\end{lemma}

\subsection{Mixed-coefficient commutator series}\label{subsec:commutator series}
To show linear independence of knots modulo $n.5$-solvability, we will use Amenable Signature Theorem~\ref{thm:obstruction}, which is available for the $\rhot$-invariants associated to a representation to a group which is amenable and in $D(R)$. In this subsection, we give examples of groups which are amenable and in $D(R)$. Namely, for a group $G$, we will construct a certain subnormal series $\{\cP^kG\}$ of $G$ such that $G/\cP^kG$ are amenable and  in $D(R)$.
\begin{definition}\cite[Definition 4.1]{Cha:2010-1}\label{def:commutator series} Let $G$ be a group and $\cP=(R_0,R_1,\ldots)$ be a sequence of rings with unity. The {\it $\cP$-mixed-coefficient commutator series $\{\cP^kG\}$} of $G$ is defined inductively as follows: let $\cP^0G:=G$. For a nonnegative integer $k$, we define
\[
\cP^{k+1}G := \Ker\left\{ \cP^kG\to \frac{\cP^kG}{[\cP^kG,\cP^kG]}\to  \frac{\cP^kG}{[\cP^kG,\cP^kG]}\otimesover\Z R_k  \right\}.
\] 
	
\end{definition}

We note that $ \cP^kG/[\cP^kG,\cP^kG]\cong H_1(G;\Z[G/\cP^kG])$ and $(\cP^kG/[\cP^kG,\cP^kG])\otimes_\Z R_k\cong H_1(G;R_k[G/\cP^kG])$. For example, if $R_k=\Z$ for all $k$, then $\cP^kG = G^{(k)}$, the $k$-th derived group of $G$. Also if $R_k=\Q$ for all $k$, then $\cP^kG=G^{(k)}_r$, the $k$-th rational derived group of $G$. Note that $G^{(k)}\subset \cP^kG$ for all $\cP$ and $k$, and the group $\cP^kG/\cP^{k+1}G$ injects into $H_1(G;R_k[G/\cP^kG])$.

Using $\cP$-mixed-coefficient commutator series, we obtain groups which are amenable and in $D(R)$ as below.

\begin{lemma}\label{lem:amenable and D(R)}\cite[Lemma 4.3]{Cha:2010-1}
Let $G$ be a group and $n$ a nonnegative integer. Let $\cP=(R_0, R_1, \ldots)$ be a sequence of rings with unity such that for each $k< n$, every integer relatively prime to $p$ is invertible in $R_k$. Then, for each $k\le n$, the group $G/\cP^kG$ is amenable and lies in $D(\Z_p)$.
\end{lemma}

Later in the proof of Theorem~\ref{thm:main}, we will use a $\cP$-mixed-coefficient commutator series where $\cP=(R_0.R_1,\ldots, R_n)$ with $R_i=\Q$ for $0\le i\le n-1$ and $R_n=\Z_p$ for some prime $p$ and a representation to the group $G/\cP^{n+1}G$ for some group $G$. In this case, $G/\cP^{n+1}G$ has a subgroup $\cP^nG/\cP^{n+1}G$ which injects into a $\Z_p$-vector space $H_1(G;\Z_p[G/\cP^nG])$. Therefore $G/\cP^{n+1}G$ has $p$-torsion elements and it is not a PTFA group. 

\section{Construction of examples}\label{sec:construction}
In this section, we discuss how to construct a knot bounding a grope of height $n+2$ in $D^4$. We will construct such a knot using a process called {\it infection} or {\it satellite construction}. Let $K$ be a knot in $S^3$ and $n$ a positive integer. Let $\eta_1,\eta_2,\ldots, \eta_m$ be disjoint simple closed curves in $E(K)$ such that  $\eta_i\in\pi_1E(K) ^{(n)}$ for all $i$. Suppose $\eta_i$ form an unlink in $S^3$. Let $J_1, \ldots, J_m$ be knots. For each $1\le i\le m$, remove the open tubular neighborhood of $\eta_i$ in $S^3$ and  glue in the exterior of $J_i$ by identifying their common boundaries using an orientation-reversing homeomorphism in such a way that the meridian (resp. the longitude) of $\eta_i$ is identified with the longitude (resp. the meridian) of $J_i$. The resulting 3-manifold is homeomorphic to $S^3$, and now the knot $K$ becomes a new knot in this $S^3$. We denote this knot by $K(\eta_i;J_i)$ or $K(\eta_1,\ldots, \eta_m; J_1,\ldots, J_m)$ and say that it is obtained by {\it infecting $K$ by $J_i$ along $\eta_i$}. In the case that $J_i=J$ for some knot $J$ for all $i$, we simply write $K(\eta_1,\ldots, \eta_m;J)$ or $K(\eta_i;J)$. We call $K$, $\eta_i$, and $J_i$ the {\it seed knot}, the {\it axes}, and the {\it auxiliary knots} (or {\it infection knots}), respectively. Fore more details, we refer the reader to \cite{Cochran-Orr-Teichner:2002-1}. 

Roughly speaking, a grope of height $n+2$ bounded by $K(\eta_i;J_i)$ is constructed by `stacking' gropes of height 2 bounded by $J_i$ and gropes of height $n$ bounded by $\eta_i$. Construction of a grope under infection was investigated in \cite{Cochran-Teichner:2003-1} and  later in \cite{Horn:2010-1}. Afterwards, a more systematic way of construction was given in \cite{Cha:2012-1,Cha-Kim:2016-1}. For instance, the following theorem is implicitly proved in \cite[Definition~4.4]{Cha-Kim:2016-1}. In the following, a {\it capped grope} is a grope with disks attached along symplectic basis curve on the top stage surfaces of the grope. 

\begin{theorem}\label{thm:grope of height n+2} \cite{Cha-Kim:2016-1} Let $K$ be a knot and let $\eta_i$, $1\le i\le m$, be curves in $S^3\setminus K$ which form an unlink in $S^3$. Suppose that $\eta_i$ bound disjoint capped gropes of height $n$ in $S^3$ which do not meet $K$ except for the caps. For each $i$ with $1\le i\le m$, suppose that $J_i$ is a knot which bounds a grope of height 2 in $D^4$. Then the knot $K(\eta_i;J_i)$ bounds a grope of height $n+2$ in $D^4$. 
\end{theorem}

Therefore we need to find axes $\eta_i$ bounding gropes of height $n$. The following lemma shows how to find such axes. 

\begin{lemma}\cite[Lemma~3.9]{Cochran-Teichner:2003-1}\label{lem:axes bounding grope of height n}
Let $K$ be a knot and $\eta_i$, $1\le i\le m$, curves in $S^3\setminus K$ which form an unlink in $S^3$. Suppose $\eta_i\in \pi_1(S^3\setminus K)^{(n)}$. Then, there exist capped gropes $G_i$ of height $n$ disjointly embedded in $S^3$ such that $G_i$ do not meet $K$ except for the caps and for each $i$ the grope $G_i$ is bounded by a knot in the homotopy class of $\eta_i$. 
\end{lemma}

Theorem~\ref{thm:grope of height n+2} and Lemma~\ref{lem:axes bounding grope of height n} yield the following corollary immediately.

\begin{corollary}\label{cor:knot bounding a grope of height n+2}
Let $K$ be a knot and let $\eta_i$, $1\le i\le m$, be curves in $S^3\setminus K$ which form an unlink in $S^3$. Suppose $\eta_i\in \pi_1(S^3\setminus K)^{(n)}$. For each $i$ with $1\le i\le m$, suppose $J_i$ is a knot bounding a grope of height 2 in $D^4$. Then we can homotope $\eta_i$ such that the knot $K(\eta_i;J_i)$ bounds a grope of height $n+2$ in $D^4$. 
\end{corollary}

We discuss further how to choose $\eta_i$ and $J_i$. Namely, to prove Theorem~\ref{thm:main} we will need to construct infinitely many knots bounding a grope of height $n+2$ in $D^4$ which are linearly independent modulo $\F_{n.5}$. For that purpose, we need to make specific choices for $\eta_i$ and $J_i$. 

The following theorem is a generalization of \cite[Theorem 5.13]{Cochran-Kim:2004-1}, and in Section~\ref{sec:infinite rank subgroup} it will be used for the choice of axes $\eta_i$ in the construction of the generating knots of the subgroups in Theorem~\ref{thm:main}.

\begin{theorem}\label{thm:eta_i} Let $n\ge 1$ be an integer. Let $K$ be a knot with nontrivial Alexander polynomial $\Delta_K(t)$. Suppose $\deg \Delta_K(t) > 2$ if $n> 1$. Let $\Sigma$ be a Seifert surface for $K$. Then there exists an unlink $\{\eta_1,\ldots, \eta_m\}$ in $S^3$ which does not meet $\Sigma$ and satisfies the following:
	\begin{enumerate}
		\item For all $i$, $\eta_i\in \pi_1M(K)^{(n)}$ and the $\eta_i$ bound capped gropes of height $n$ which are disjointly embedded in $S^3\setminus K$. Here, the caps are allowed to intersect $K$. 
		\item Let $\cP=(R_0,R_1,\ldots, R_n)$ where $R_i=\Q$ for $0\le i\le n-1$ and $R_n=\Z_p$ where $p$ is a prime greater than the top coffeicient of $\Delta_K(t)$. Then, for each $n$-cylinder $W$ with $M(K)$ as one of its boundary components, there exists some $\eta_i$ such that $j_*(\eta_i)\notin \cP^{n+1}\pi_1W$ where $j_*\colon \pi_1M(K)\to \pi_1W$ is the inclusion-induced homomorphism. 
	\end{enumerate}   
\end{theorem}
\noindent The above theorem is the most technical part of the paper, and its proof is postponed to the end of  Subsection~\ref{subsec:algebraic n-solution}. Theorem~\ref{thm:eta_i} is significant since it shows that there is a finite set of $\eta_i$ in $S^3\setminus K$ (in fact, in $S^3\setminus \Sigma$) which satisfies the nontriviality property in Theorem~\ref{thm:eta_i}~(2) for {\it every $n$-cylinder} for $M(K)$: if one needs to find a finite set of $\eta_i$ satisfying the nontriviality property for {\it a specific choice of an $n$-cylinder}, it can be more easily done (see Theorem~\ref{thm:nontrivial}.)
\begin{figure}[H]
	\begin{tikzpicture}[x=1bp,y=1bp]
	\small
	%	\drawguidegrid{0,0}{300,200}
	\node [anchor=south west, inner sep=0mm] {\includegraphics[scale=1]{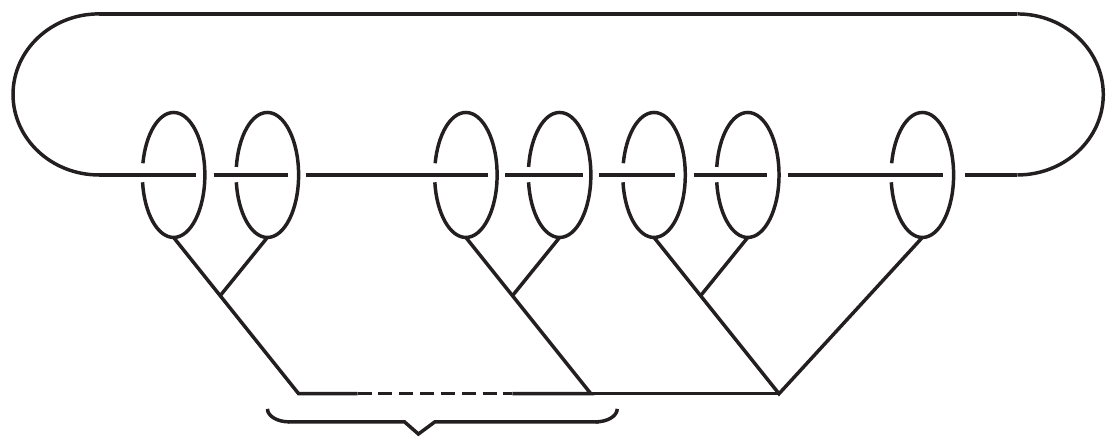}};
	\node [left] at (5,100) {$U$};
	\node [left] at (60,45) {$T$};
	\node [below] at (120,0) {$m$};
	\end{tikzpicture}
	\caption{A clasper surgery description of the knot $P_m$}
	\label{figure:Knot_P_m}
\end{figure}
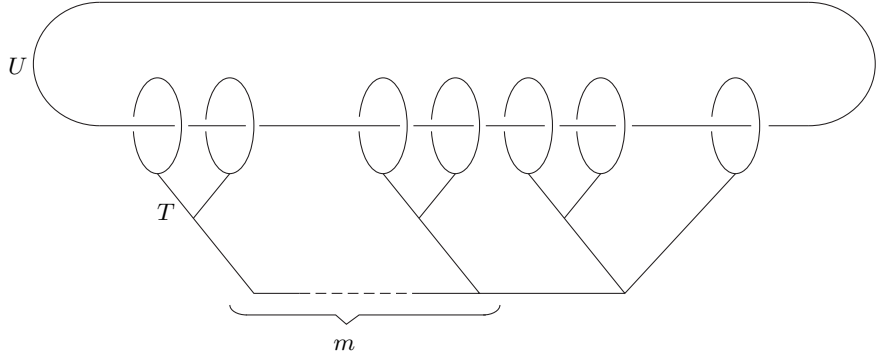

We will need the following lemma for the choice of infection knots $J_i$ when we show linear independence of knots $K(\eta_1,\ldots, \eta_m;J_i)$, $i\ge 1$, in the proof of Theorem~\ref{thm:main}. 

\begin{lemma}\label{lemma:J_0^i}\cite[Proposition 3.4]{Jang:2015-1}
	For an arbitrary constant $C$, there exists a sequence of knots $J_1, J_2, \ldots,$ and an increasing sequence of odd primes $p_1,p_2,\ldots, $ which satisfy the following: let $\omega_i := e^{2\pi\sqrt{-1}/p_i}$. 
	\begin{enumerate}
		\item Each $J_i$ bounds a grope of height 2 in $D^4$, 
		\item $\sum_{r=0}^{p_i-1}\sigma_{J_i}(\omega_i^r) > C$,
		\item $\sum_{r=0}^{p_j-1}\sigma_{J_i}(\omega_j^r) =0$ for $j<i$.
	\end{enumerate}	
\end{lemma}
In \cite[Propositoin~3.4]{Jang:2015-1}, the knots $J_i$ were constructed to satisfy more conditions such as $\int_{S^1} \sigma_{J_i}(\omega)\,d\omega=0$ to obtain the knots $K(\eta_1,\ldots, \eta_m;J_i)$ for which the $\rhot$-invariants associated to a PTFA representation vanish. We do not need this property in this paper, and for our purpose we can use simpler $J_i$. That is, the proof of \cite[Propositoin~3.4]{Jang:2015-1} tells us that in Lemma~\ref{lemma:J_0^i}, for each $i$ we can take $J_i$ to be a connected sum of $N$ copies of $P_{m_{i+1}}\# (-P_{m_i})$ where $N$ is an integer bigger than $C/2$, $P_m$ is a knot whose clasper surgery description is given in Figure~\ref{figure:Knot_P_m}, and $m_1,m_2,\ldots,$ is a certain increasing sequence of positive integers. In Figure~\ref{figure:Knot_P_m}, $P_m$ is obtained from the unknot $U$ by performing clasper surgery along the tree $T$. The knot $P_m$ for $m=1$ was given in \cite[Figure~3.7]{Cochran-Teichner:2003-1}, and $P_m$ for $m>1$ were given in \cite{Horn:2010-1}. The surgery descriptions of $P_m$ are given in \cite[Figure~3.6]{Cochran-Teichner:2003-1} and \cite[Figure~3]{Horn:2010-1}.

\section{Infinite rank subgroups of filtrations}\label{sec:infinite rank subgroup}
In this section, we prove Theorem~\ref{thm:main},  which will be a direct consequence of Theorems~\ref{thm:refined main theorem-1} and \ref{thm:refined_main_theorem-2}. First, we construct the generating knots of the desired subgroups in Theorem~\ref{thm:main}.
Let $n> 1$. Let $K$ be a slice knot with $\deg \Delta_K(t)>2$. Let $\{\eta_1,\ldots \eta_m\}$ be an unlink in $S^3$ lying in $E(K)$ given by Theorem~\ref{thm:eta_i}. By Corollary~\ref{cor:knot bounding a grope of height n+2}, we can homotope $\eta_i$ such that $K(\eta_i;J)$ bounds a grope of height $n+2$ in $D^4$ for every knot $J$ which bounds a grope of height 2 in $D^4$. By Lemma~\ref{lem:universal bound}, there exists a constant $C$ such that $|\rhot(M(K),\phi)|<C$ for all homomorphisms $\phi\colon \pi_1M(K)\to G$ for every group $G$. Choose a sequence of knots $J_1, J_2, \ldots ,$ using Lemma~\ref{lemma:J_0^i} with this constant $C$ such that the prime $p_1$ is greater than the top coefficient of $\Delta_K(t)$. Finally, for each $i\ge 1$ define $K_i:=K(\eta_1,\ldots, \eta_m;J_i)$. (Therefore, the choice of $\eta_1,\ldots,\eta_m$ is independent of $K_i$.)

\subsection{Linear independence of examples}\label{subset:linear independence}

In this subsection, in Theorem~\ref{thm:refined main theorem-1} we show that the knots $K_i$ defined as above satisfy Theorems~\ref{thm:main}(1) and (2). To prove Theorem~\ref{thm:refined main theorem-1}, we will need the following lemma. 
Recall that for a 4-manifold $W$ with a homomorphism $\phi\colon \pi_1W\to G$ , we let $S_G(W):=\lsign(W)-\sign(W)$. 
\begin{lemma}\label{lem:building block} 
Let $K$ be a slice knot. Let $\{\eta_1.\eta_2,\ldots,\eta_m\}$ be an unlink in $E(K)$ such that $\eta_i \in \pi_1E(K)^{(n)}$ for all $1\le i\le m$. Let $M:=M(K(\eta_i;J_i))$. Suppose $J_i$ is a knot with vanishing Arf invariant for $1\le i\le m$. Then, letting $\Z_\infty:=\Z$, we have the following.
	\begin{enumerate}
	\item There exists an $n$-solution $W$ for $M$  satisfying the following: suppose $\phi\colon \pi_1W\to G$ is a homomorphism where $G$ is an amenable group lying in Strebel's class $D(R)$ for some ring $R$. Then, $S_G(W)=\sum_{i=1}^m\rhot(M(J_i), \phi_i)$ where for each $i$ the map $\phi_i\colon \pi_1M(J_i)\to \Z_{d_i}$ is a surjective homomorphism sending the meridian of $J_i$ to $1\in \Z_{d_i}$ with $d_i$ the order of $\phi(\eta_i)$ in $G$. 
	\item There exists an $n$-cylinder $V$ with $\partial V=M(K)\coprod (-M)$ satisfying the following: suppose $\phi\colon \pi_1V\to G$ is a homomorphism where $G$ is an amenable group lying in Strebel's class $D(R)$ for some ring $R$. Then, $S_G(V)=-\sum_{i=1}^m\rhot(M(J_i), \phi_i)$ where for each $i$ the map $\phi_i\colon \pi_1M(J_i)\to \Z_{d_i}$ is a surjective homomorphism sending the meridian of $J_i$ to $1\in \Z_{d_i}$ with $d_i$ the order of $\phi(\eta_i)$ in $G$. 

	\end{enumerate}

\end{lemma}
\begin{proof}
Part~(1) is essentially due to \cite[Proposition 4.4]{Cha:2010-1} and its proof, noticing that
\[
\rhot(M,\phi) = \rhot(M(K), \phi) + \sum_{i=1}^m\rhot(M(J_i), \phi_i)
\]
where, by abuse of notation, $\phi$ also denotes the restriction of the map $\phi\colon \pi_1W\to G$ to the corresponding subspace. 

We prove Part~(2). Since each $J_i$ has vanishing Arf invariant, it is $0$-solvable. Let $W_i$ be a $0$-solution for $J_i$. By doing surgery along $\pi_1W_i^{(1)}$ if necessary, we may assume that $\pi_1W_i\cong \Z$, generated by the meridian of $J_i$. 
	
	Let $V=M(K)\times [0,1]\cup (\coprod_{i=1}^m-W_i)$ where each $-W_i$ is attached to $M(K)\times [0,1]$ by identifying the solid torus $M(J_i)-E(J_i)\subset \partial W_i$ with the tubular neighborhood of $\eta_i\times 0\subset M(K)\times 0$ in such a way that the 0-linking longitude of $\eta_i$ is identified with the meridian of $J_i$ and the meridian of $\eta_i$ is identified with the 0-linking longitude of $J_i$. Then $\partial V=M(K)\amalg (-M)$ and an inclusion from a boundary component $M(K)$ or $M$ to $V$ induces an isomorphism on first homology. Using Mayer-Vietoris sequences, one can also show that $0$-Lagrangians and $0$-duals of $W_i$ give rise to an $n$-Lagrangian and its $n$-dual of $V$. This shows that $V$ is an $n$-cylinder. 
	
As in \cite[Lemma 2.3]{Cochran-Harvey-Leidy:2009-1}, one can obtain that 
	\[
	\rhot(M(K),\phi_K) - \rhot(M,\phi_M) = -\sum_{i=1}^m \rhot(M(J_i),\phi_i)
	\]
	where $\phi_K$, $\phi_M$, and $\phi_i$ are the restrictions of $\phi $ to the corresponding subspaces. By the definition of $\rhot$-invariants, we have $\rhot(M(K),\phi_K) - \rhot(M,\phi_M) =  S_G(V)$. Note that since $\phi_i$ factors through $\pi_1W_i\cong\Z$, the image of $\phi_i$ is the abelian subgroup in $G$ generated by the image of the meridian of $J_i$. Since the meridian of $J_i$ is identified with the 0-linking longitude of $\eta_i$ in $M(K)$, by the subgroup property of $\rhot$-invariants (see\cite[p.108]{Cochran-Orr-Teichner:2002-1}) we may assume that the map $\phi_i\colon \pi_1M(J_i)\to \Z_{d_i}$ is a surjective homomorphism sending the meridian of $J_i$ to $1\in \Z_{d_i}$ with $d_i$ the order of $\phi(\eta_i)$ in $G$. 
\end{proof}

\begin{theorem}\label{thm:refined main theorem-1} Let $n\ge 2$ be an integer and let $K$ be a slice knot with $\deg \Delta_K(t) > 2$. Let $K_i$ be the knots obtained from $K$ as in the beginning of Section~\ref{sec:infinite rank subgroup}. Then, $K_i\in \G_{n+2}$ for all $i$, and $K_i$ are linearly independent modulo $\F_{n.5}$. Moreover, letting $G:=\pi_1(S^3\setminus K)$ and $G_i:=\pi_1(S^3\setminus K_i)$, we can construct $K_i$ such that for each $i$ there is an  isomorphism $G_i/G_i^{(n+1)}\to G/G^{(n+1)}$ preserving peripheral structures.
\end{theorem}

\begin{proof}
	First, we prove the last part. Since there is a degree 1 map from $E(J_i)$ to $E(\mbox{unknot})$, for each $i$, there is a degree 1 map $f_i\colon E(K_i) \to E(K)$ relative to the boundary such that $f_i$ is the identity outside the copies of $E(J_i)$. Since $\eta_\ell\in \pi_1M(K)^{(n)}$ for all $\ell$, by \cite[Theorem~8.1]{Cochran:2002-1} the $f_i$ induces an isomorphism $G_i/G_i^{(n+1)}\to G/G^{(n+1)}$ which preserves peripheral structures. 
	
The $K_i$ bound a grope of height $n+2$ in $D^4$ by Corollary~\ref{cor:knot bounding a grope of height n+2} and Lemma~\ref{lemma:J_0^i}, hence $K_i\in \cG_{n+2}$. Let $J:=\#_ia_iK_i$ where $a_i\in \Z$, a nontrivial connected sum of finitely many copies of $K_i$ and their inverses $-K_i$. To prove the theorem it suffices to show that $J$ is not $n.5$-solvable. 
	
	Suppose $J$ is $n.5$-solvable. We may assume $a_1\ne 0$, and furthermore, by taking the inverse of $K_1$ if necessary, we may assume $a_1>0$. Note that the $J_i$ have vanishing Arf invariant since they bound a grope of height 2 in $D^4$ and hence $0$-solvable. We construct building blocks for a certain 4-manifold as follows.
	\begin{enumerate}
		\item Let $V$ be an $n.5$-solution for $M(J)$. 
		\item Let $E$ be the standard cobordism between $M(J)$ and $\amalg_ia_iM(K_i)$ as constructed (with the name $C$) in \cite[p.113]{Cochran-Orr-Teichner:2002-1}. We may assume $\partial E = (\amalg_ia_iM(K_i))\amalg (-M(J))$.
		\item Let $V_1$ be an $n$-cylinder with $\partial V=M(K) \amalg(- M(K_1))$ as given in Lemma~\ref{lem:building block}~(2).
		\item Let $W_i$ be an $n$-solution for $M(K_i)$ as given in Lemma~\ref{lem:building block}~(1).
	\end{enumerate}

	Let $b_1:=a_1-1$ and $b_i=|a_i|$ for $i\ge 2$. For $1\le r\le b_i$, let $W_i^r$ be a copy of $-\epsilon W_i$ where $\epsilon_i=a_i/|a_i|$. Now we define
	\[
	W=V\bigcup_{\partial_-E} E \bigcup_{\partial_+ E} \left(V_1\coprod\left(\coprod_{i}\coprod_{r=1}^{b_i}W_i^r\right)\right)
	\]
	where $\partial_+E = \coprod_ia_iM(K_i)$ and $\partial_-E=M(J)$. See Figure~\ref{figure:Cobordism_W}. Note that $\partial W=M(K)$. Using Mayer-Vietoris sequences, one can easily show that the $n$-Lagrangians and $n$-duals of $V_1$, $W_i^r$, and $V$ form an $n$-Lagrangian and its $n$-dual for $W$, and therefore $W$ is an $n$-solution for $M(K)$.
	
	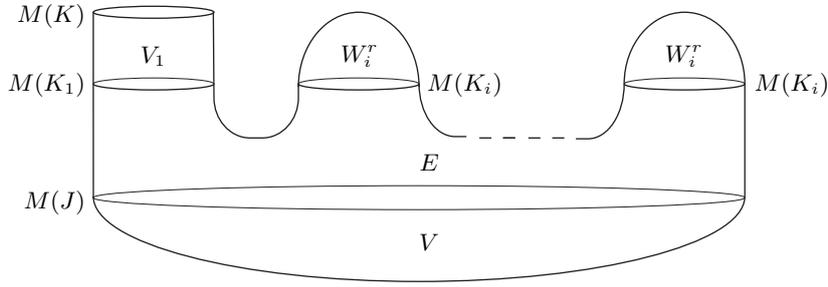
\begin{figure}[H]
		\begin{tikzpicture}[x=1bp,y=1bp]
		\small
		%\drawguidegrid{0,0}{300,200}
		\node [anchor=south west, inner sep=0mm] {\includegraphics[scale=0.5]{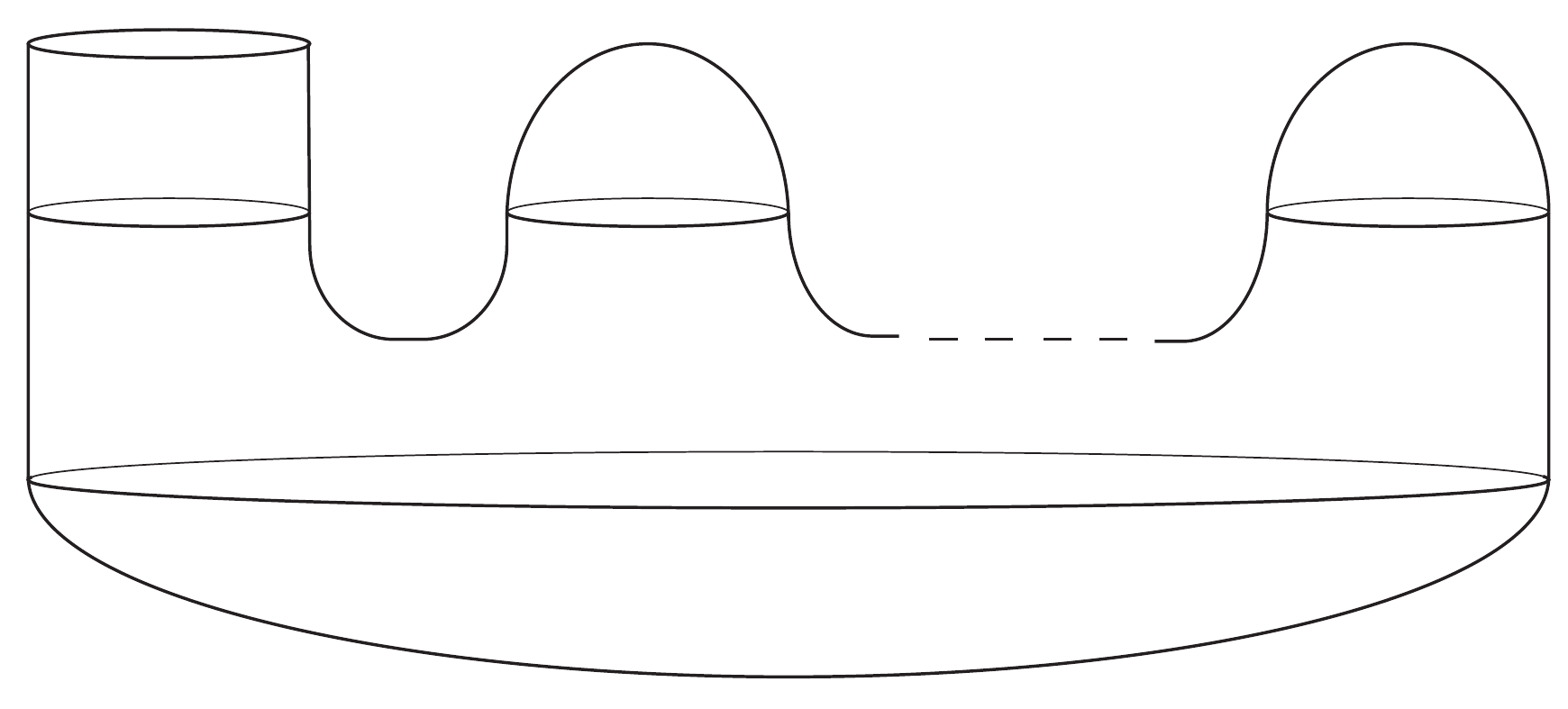}};
		\node at (130,20) {$V$};
		\node at (130,50) {$E$};
		\node at (27,90) {$V_1$};
		\node at (104,90) {$W_i^r$};
		\node at (225,90) {$W_i^r$};
		\node [left] at (5,35) {$M(J)$};
		\node [left] at (5,79) {$M(K_1)$};
		\node [left] at (5,105) {$M(K)$};
		\node [right] at (126,79) {$M(K_i)$};
		\node [right] at (248,79) {$M(K_i)$};
		\end{tikzpicture}
		\caption{Cobordism $W$}
		\label{figure:Cobordism_W}
	\end{figure}

	Let $\cP:=(R_0,R_1,\ldots, R_n)$ where $R_i=\Q$ for $i\le n-1$ and $R_n=\Z_{p_1}$. Using Definition~\ref{def:commutator series} we obtain $\cP^j\pi_1 W$ for $1\le j\le n+1$, which are subgroups of $\pi_1W$. Let $G:=\pi_1W/\cP^{n+1}\pi_1W$, which is amenable and lies in $D(\Z_{p_1})$ by Lemma~\ref{lem:amenable and D(R)},  and let $\phi\colon \pi_1W\to G$ be the quotient map. For convenience, we denote a restriction of $\phi$ by $\phi$ as well. 
	
	Since $\partial W=M(K)$, we have $S_G(W)=\rhot(M(K),\phi)$. On the other hand, by Novikov additivity, we have
	\[
	S_G(W)=S_G(V) + S_G(E) + S_G(V_1) + \sum_i\sum_{r=1}^{b_i}S_G(W_i^r).
	\]
	
	Here, $S_G(V)=0$ by Amenable Signature Theorem~\ref{thm:obstruction}. 
	
	Also, $S_G(E)=0$ since $\Coker\{H_2(\partial_-E)\to H_2(E)\}=0$ (see \cite[Lemma~4.2]{Cochran-Orr-Teichner:2002-1}\cite[Lemma~2.4]{Cochran-Harvey-Leidy:2009-1}). 
	
	By Lemma~\ref{lem:building block}(2), 
	\[
	S_G(V_1)=-\sum_{i=1}^m\rhot(M(J_1), \phi_i)
	\] 
	where $\phi_i\colon \pi_1M(J_1)\to \Z_{d_i}$ is a surjective homomorphism such that $d_i$ is the order of $\phi(\eta_i)$ in $G$ and $\phi_i$ sends the meridian of $J_1$, which is the 0-linking longitude of $\eta_i$, to $1\in \Z_{d_i}$.  Since $\eta_i\in \pi_1M(K)^{(n)}$, one can see that $\phi(\eta_i)$ lies in $\cP^n\pi_1W/\cP^{n+1}\pi_1W$. Since $\cP^n\pi_1W/\cP^{n+1}\pi_1W$ injects into $H_1(\pi_1W;\Z_{p_1}[\pi_1W/\cP^n\pi_1W])$, which is a $\Z_{p_1}$-vector space, we have $d_i=1\mbox{ or } p_1$. Then, by Lemma~\ref{lem:computation of rho-invariant}, $\rhot(M(J_1),\phi_i)=0$ if $d_i=1$, and $\rhot(M(J_1),\phi_i)=\sum_{r=0}^{p_1-1}\sigma_{J_1}(e^{2\pi r\sqrt{-1}/p_1})$ if $d_i=p_1$. Furthermore, since the $\eta_\ell$ $(1\le \ell\le m)$ were chosen using Theorem~\ref{thm:eta_i}, there exists some $\eta_i$ such that $\phi(\eta_i)\ne e$ in $G$. Therefore $d_i=p_1$ for some $i$ and we have 
	\[
	S_G(V_1) \le -\sum_{r=0}^{p_1-1}\sigma_{J_1}(e^{2\pi r\sqrt{-1}/p_1}).
	\]
	
	We compute $S_G(W_i^r)$. By Lemma~\ref{lem:building block}(1), we have 
	\[
	S_G(W_i)= \sum_{j=1}^m\rhot(M(J_i), \phi_j).
	\] 
	When $i=1$, $W_1^r=-W_1$ and hence 
	\[
	S_G(W_1^r)= - \sum_{r=0}^{p_1-1}\sigma_{J_1}(e^{2\pi r\sqrt{-1}/p_1}) \mbox{ or } 0,
	\]
	computed as above. When $i\ge 2$, similarly to the case $i=1$, we have 
	\[
	S_G(W_i^r)= \pm\sum_{r=0}^{p_1-1}\sigma_{J_i}(e^{2\pi r\sqrt{-1}/p_1}) \mbox{ or } 0.
	\] 
	But by Lemma~\ref{lemma:J_0^i}(3), we have $S_G(W_i^r)=0$ for $i\ge 2$.
	
	Summing up the above computations, we obtain that 
	\[
	S_G(W) \le -\sum_{r=0}^{p_1-1}\sigma_{J_1}(e^{2\pi r\sqrt{-1}/p_1})
	\]
	 and therefore we have $|S_G(W)|> C$ by our choice of $J_1$ and Lemma~\ref{lemma:J_0^i}(2). But from $S_G(W)=\rhot(M(K),\phi)$ and our choice of $C$, we have $|S_G(W)|< C$, which is a contradiction. 
\end{proof}

\subsection{Distinction from the knots constructed via iterated doubling operators}\label{subsec:distinction}

The purpose of this subsection is to prove Theorem~\ref{thm:refined_main_theorem-2} below. We note that in \cite{Horn:2010-1} it was shown that a nontrivial combination of the knots generating the infinite rank subgroup of $\cG_{n+2}/\cG_{n+2.5}$ in \cite{Horn:2010-1} is nontrivial modulo $\F_{n.5}$. That is, the infinite rank subgroup of $\cG_{n+2}/\cG_{n+2.5}$ in \cite{Horn:2010-1} injects into $\F_n/\F_{n.5}$ under the quotient map $\cC/\cG_{n+2.5}\to \cC/\F_{n.5}$.
\begin{theorem}\label{thm:refined_main_theorem-2}
Let $n\ge 2$ be an integer and let $K_i$ be the knots in Theorem~\ref{thm:refined main theorem-1}. In $\F_n/\F_{n.5}$, the infinite rank subgroup generated by the $K_i$ trivially intersects the infinite rank subgroups of $\F_n/\F_{n.5}$ in \cite{Cochran-Harvey-Leidy:2009-1}, \cite{Horn:2010-1}, \cite{Cochran-Harvey-Leidy:2009-2}, and \cite{Cha:2010-1}. In particular, in $\cG_{n+2}/\cG_{n+2.5}$, the infinite rank subgroup  generated by $K_i$ trivially intersects the infinite rank subgroup of $\cG_{n+2}/\cG_{n+2.5}$ in \cite{Horn:2010-1}.
\end{theorem}

We give a proof at the end of this subsection. The proof of the above theorem is based on ideas in \cite[Section 9]{Cochran-Harvey-Leidy:2009-2}. Using the terms in \cite[Section 9]{Cochran-Harvey-Leidy:2009-2}, our knots $K_i$ can be considered as {\it generalized COT knots}, and on the other hand the knots in \cite{Cochran-Harvey-Leidy:2009-1}, \cite{Horn:2010-1}, \cite{Cochran-Harvey-Leidy:2009-2}, and \cite{Cha:2010-1} are called {\it CHL knots}. Therefore, to prove Theorem~\ref{thm:refined_main_theorem-2}, we need to extend the results in \cite[Section 9]{Cochran-Harvey-Leidy:2009-2}, which are about distinguishing COT knots from CHL knots, to the case of generalized COT knots. We will do this by showing the following: 1) For each $n>0$ and a prime $p$, we define a subset $\fpn$ of the set of (isotopy classes of) knots in $S^3$; 2) We show that if $K$ is a nontrivial linear combination of the $K_i$ in Theorem~\ref{thm:refined main theorem-1} which generate an infinite rank subgroup in $\F_n/\F_{n.5}$ (and in $\cG_{n+2}/\cG_{n+2.5}$), we have $K\notin \fpnh$ for some prime $p$; 3) We show that if a knot $K$ is concordant to a nontrivial linear combination of CHL knots, then $K\in \fpnh$ for all prime $p$.

First, for each $n> 0$, we define a subset $\fpn$, which generalizes the subset $\F^{cot}_n$ defined in \cite{Cochran-Harvey-Leidy:2009-2}. Fix an integer $n\ge 0$ and a prime $p$. Let $G$ be a group such that $H_1(G)\cong \Z$. Since $G^{(1)}/G^{(n)}_r$ is a PTFA group \cite[Proposition~2.1]{Harvey:2006-1},  $\Z_p[G^{(1)}/G^{(n)}_r]$ embeds into its skew quotient field, say $\KK$ (see Lemma~\ref{lem:quotient field}). Since $H_1(G)=G/G^{(1)}\cong \Z$, the group ring $\Z_p[G/G^{(n)}_r]$ can be embedded into the noncommutative Laurent polynomial ring $\KK[t^{\pm 1}]$, which is a noncommutative PID. Now we define 

\[
G^{(n+1)}_{cot,p} := \Ker\left\{ G^{(n)}_r\to \frac{G^{(n)}_r}{[G^{(n)}_r,G^{(n)}_r]}\to  \frac{G^{(n)}_r}{[G^{(n)}_r,G^{(n)}_r]}\otimesover{\Z[G/G^{(n)}_r]} \KK[t^{\pm 1}] \right\}.
\]
Note that $G^{(n+1)}\subset G^{(n+1)}_{cot,p}$ for all prime $p$ and 
\[
\frac{G^{(n)}_r}{[G^{(n)}_r,G^{(n)}_r]}\otimesover{\Z[G/G^{(n)}_r]} \KK[t^{\pm 1}] \cong H_1(G;\KK[t^{\pm 1}])
\]
since $\KK[t^{\pm 1}]$ is a flat left module over $\Z[G/G^{(n)}_r]$ by \cite[Proposition II.3.5]{Stenstrom:1975}.

For a 4-manifold $W$ with $H_1(W)\cong \Z$, let $\pi:=\pi_1W$. For each $n>0$ and a prime $p$, there is the equivariant intersection form 
\[
\lambda_n^p\colon H_2(W;\Z[\pi/\pi^{(n)}_{cot,p}]) \times H_2(W;\Z[\pi/\pi^{(n)}_{cot,p}]) \to \Z[\pi/\pi^{(n)}_{cot,p}]. 
\] 
\begin{definition}\label{def:p-F_n} Let $n$ be a positive integer and $p$ a prime. A knot $K$ is {\it $(n,p)$-solvable via $W$} if there exists a compact connected 4-manifold $W$ with boundary $M(K)$ satisfying the following:
let $\pi:=\pi_1W$ and $r:=\frac12 \rank_\Z H_2(W)$.
\begin{enumerate}
	\item The inclusion $M\to W$ induces an isomorphism $H_1(M)\to H_1(W)$.
	\item There exist  $x_1, x_2, \ldots, x_r$ and $y_1, y_2, \ldots, y_r$ in $H_2(W;\Z[\pi/\pi^{(n)}_{cot,p}])$ such that $\lambda_n^p(x_i,x_j)=0$ and $\lambda_n^p(x_i,y_j)=\delta_{ij}$ for $1\le i,j\le r$.
\end{enumerate}	
We say $K$ is {\it $(n.5,p)$-solvable via $W$} if the following additional condition is satisfied:
\begin{enumerate}
	\item[(3)] There exist $\tilde{x}_1,\tilde{x}_2,\ldots, \tilde{x}_r$ in $H_2(W;\Z[\pi/\pi^{(n+1)}_{cot,p}])$ such that $\lambda_{n+1}^p(\tilde{x}_i, \tilde{x}_j)=0$ for $1\le i,j\le r$ and $\tilde{x}_i$ and $x_i$ are represented by the same surface for each $i$.
\end{enumerate}
We denote the set of $(n,p)$-solvable knots and $(n.5,p)$-solvable knots by $\fpn$ and $\fpnh$, respectively. The submodules generated by $x_i$ and $y_i$ are called an {\it $(n,p)$-Lagrangian} and an {\it $(n,p)$-dual}, respectively. The submodule generated by $\tilde{x}_i$ is called an {\it $(n+1,p)$-Lagrangian.} 
\end{definition}
Since  $\pi^{(n)}\subset \pi^{(n)}_{cot,p}$, if a knot $K$ is $n$-solvable, then it is $(n,p)$-solvable for all prime $p$. We note that the notion of $G^{(n)}_{cot,p}$ is not functorial. That is, for a group homomorphism $\phi\colon G\to H$, in general $\phi(G^{(n)}_{cot,p})\nsubseteq H^{(n)}_{cot,p}$. Due to this fact, the subset $\fpn$ does not descend to a subgroup of $\cC$.

The following theorem generalizes \cite[Theorem 5.2]{Cochran-Harvey-Leidy:2009-2} and \cite[Theorem 3.2]{Cha:2010-1}, and it can be proved using the same arguments as in the proof of \cite[Theorem 3.2]{Cha:2010-1}. Therefore, we give only a sketch of the proof.
\begin{theorem} \label{thm:vanishing rho-invariant-(n,p)}
Let $p$ be a prime and $n$ a positive integer. Let $K$ be a knot which is $(n.5, p)$-solvable via $W$. Let $G$ be an amenable group lying in Strebel's class $D(\Z_p)$. Let $\pi:=\pi_1W$ and suppose we are given a homomorphism
	 \[
	 \phi\colon \pi_1M(K)\to \pi\to \pi/\pi^{(n+1)}_{cot,p}\to G
	 \] 
	 which sends the meridian of $K$ to an infinite order element in $G$. Then $\rhot(M(K),\phi)=0$.
\end{theorem}
\begin{proof}[Sketch of proof]  We have $\rhot(M(K),\phi)=S_G(W)=\lsign_G(W)-\sign(W)$, and we will show that $S_G(W)=0$. From Definition~\ref{def:p-F_n}, we have an $(n+1,p)$-Lagrangian in $H_2(W;\Z[\pi/\pi^{(n+1)}_{cot,p}])$ generated by $\tilde{x}_1,\tilde{x}_2,\ldots, \tilde{x}_r$ and its $n$-dual in $H_2(W;\Z[\pi/\pi^{(n)}_{cot,p}])$ generated by $y_1,y_2,\ldots, y_r$ where $r=\frac12 \rank_\Z H_2(W)$. The images of $\tilde{x_i}$ and $y_i$ consist of a 0-Lagrangian and its 0-dual in $H_2(W;\Q)$, and therefore $\sign(W)=0$.
	
We show that $\lsign_G(W)=0$. Let $\N G$ be the group von Neumann algebra of $G$ and 
\[
\lambda\colon H_2(W;\N G)\times H_2(W;\N G)\to \N G
\] 
be the corresponding intersection form where the coefficients of the homology groups are twisted via $\phi$. Since $\phi\colon \pi_1M(K)\to G$ factors through $\pi/\pi^{(n+1)}_{cot,p}$, we have the induced map $H_2(W;\Z[\pi/\pi^{(n+1)}_{cot,p}])\to H_2(W;\N G)$.

Let $H$ be the submodule of $H_2(W;\N G)$ generated by the images of $\tilde{x_i}$ under this map, and $\bar{H}$ be the submodule of $H_2(W;\Z_p)$ generated by the images of $\tilde{x_i}$. Since the images of $\tilde{x_i}$ in $H_2(W;\Z_p)$ has a dual, which is the images of $y_i$, we have $\dim_{\Z_p}\bar{H}=r$. By Lemmas 3.13 and 3.14 in \cite{Cha:2010-1}, we have $\ldim H_2(W;\N G)=\dim_{\Z_p} H_2(W;\Z_p)=2r$ where $\ldim$ denotes the $\Lt$-dimension function on $\N G$-modules.  By \cite[Theorem~3.11]{Cha:2010-1}, we have
\[
\ldim H\ge \ldim H_2(W;\N G) - \dim_{\Z_2}H_2(W;\Z_p)+\dim_{\Z_p}\bar{H}=r.
\]
Now by \cite[Proposition 3.7]{Cha:2010-1}, we have $\lsign_G(W)=0$.
\end{proof}
Now we give a proof of Theorem~\ref{thm:refined_main_theorem-2}.

\begin{proof}[Proof of Theorem~\ref{thm:refined_main_theorem-2}]
Since $\cG_{n+2.5}\subset \F_{n.5}$, the last part follows from the fact that the basis knots of the infinite rank subgroup $\cG_{n+2}/\cG_{n+2.5}$ in \cite{Horn:2010-1} are linearly independent modulo $\F_{n.5}$, which is shown in the proof of \cite[Theorem~5.2]{Horn:2010-1}. 

Let $J$ be a nontrivial linear combination of the $K_i$ and let $L$ be a knot which is concordant to a linear combination of the $n$-solvable knots in \cite{Cochran-Harvey-Leidy:2009-1}, \cite{Horn:2010-1}, \cite{Cochran-Harvey-Leidy:2009-2}, or \cite{Cha:2010-1}, which are constructed using iterated doubling operators and generate an infinite rank subgroup of $\F_n/\F_{n.5}$. We will show $J\notin \F_{n.5}^{cot,p}$ for some prime $p$, but $L\in \F_{n.5}^{cot,p}$ for all prime $p$, and this will prove the theorem.
 
First, suppose $J=\#_ia_iK_i$ where $a_i\in \Z$.  As in the proof of Theorem~\ref{thm:refined main theorem-1}, we may assume $a_1>0$. We will show $J\notin \F_{n.5}^{cot,p_1}$.  Suppose to the contrary that $J$ is $(n.5,p_1)$-solvable via $V$. We follow the proof of Theorem~\ref{thm:refined main theorem-1} with the changes and observations below:
\begin{enumerate}
	\item Now $V$ is not an $n.5$-solution; it is only an $(n.5,p_1)$-solution.
	\item Use $\pi_1W^{(n+1)}_{cot,p_1}$ instead of $\cP^{n+1}\pi_1W$. We keep $\cP^n\pi_1W$ as it is, noting that $\cP^n\pi_1W = \pi_1W^{(n)}_r$.
	\item $\cP^n\pi_1W/\pi_1W^{(n+1)}_{cot,p_1}$ is a $\Z_{p_1}$-vector space: from the 
	definitions, one can see  that $\cP^{n+1}\pi_1W\subset \pi_1W^{(n+1)}_{cot,p_1}$, and therefore $\cP^n\pi_1W/\pi_1W^{(n+1)}_{cot,p_1}$ is a quotient group of $\cP^n\pi_1W/\cP^{n+1}\pi_1W$, which is a $\Z_{p_1}$-vector space.
	\item We change the definition of $G$: let $G:=\pi_1W/\pi_1W^{(n+1)}_{cot,p_1}$. Then $G$ admits a subnormal series
	\[
	\{e\} \subset G^{(n)}_r\subset G^{(n-1)}_r\subset \cdots G^{(1)}_r\subset G
	\]
	whose successive quotients are abelian and have no torsion coprime to $p_1$. Therefore, the group $G$ is amenable and lies in Strebel's class $D(\Z_{p_1})$ (see \cite[Lemma~6.8]{Cha-Orr:2009-1}).
	\item Use Theorem~\ref{thm:vanishing rho-invariant-(n,p)} instead of Theorem~\ref{thm:obstruction} to show $S_G(V)=0$.
	\item The crucial part is to show that there is some $\eta_i$ such that $\phi(\eta_i)\ne e$ for the homomorphism $\phi\colon \pi_1W\to G$. In the proof of Theorem~\ref{thm:refined main theorem-1}, we use Theorem~\ref{thm:eta_i} to show this when $G=\pi_1W/\cP^{n+1}\pi_1W$. Theorem~\ref{thm:eta_i} is proved in the end of Subsection~\ref{subsec:algebraic n-solution} in which we show that there exists some $\eta_i$ such that $\eta_i\notin \cP^{n+1}\pi_1W$. But in that proof a stronger fact is proved for $\eta_i$: this $\eta_i$ maps to a nontrivial element in $H_1(W;\KK[t^{\pm 1}])$ where $\KK$ is the skew quotient field of $\Z_p[\pi_1W^{(1)}/\pi_1W^{(n)}_r]$ (in our case $p=p_1$). This implies that this $\eta_i$ maps to a nontrivial element even in $\cP^n\pi_1W/\pi_1W^{(n+1)}_{cot,p_1}$. Therefore for this $\eta_i$, we have $\phi(\eta_i)\ne e$.
\end{enumerate}
Then, we obtain $|S_G(W)|>C$ as in the proof of Theorem~\ref{thm:refined main theorem-1}, which is a contradiction. This completes the proof that $J\notin \F_{n.5}^{cot,p_1}$.

Next, we show $L\in \F_{n.5}^{cot,p}$ for all prime $p$. The proof is similar to the one of Proposition~9.2 in \cite{Cochran-Harvey-Leidy:2009-2}. Suppose $L$ is concordant to a knot $J:=\#_i a_i J_i$ where $a_i\in \Z$ and $J_i$ are $n$-solvable knots in \cite{Cochran-Harvey-Leidy:2009-1}, \cite{Horn:2010-1}, \cite{Cochran-Harvey-Leidy:2009-2}, or \cite{Cha:2010-1}, which are constructed using iterated doubling operators and generate an infinite rank subgroup of $\F_n/\F_{n.5}$. 

We explain iterated doubling operators in more detail. Let $R$ be a slice knot and $\alpha$ be a finite set of simple closed curves $\{\eta_1,\eta_2,\ldots, \eta_m\}$ in $S^3\setminus R$ such that $\eta_i$ form an unlink in $S^3$ and $\eta_i\in \pi_1(S^3\setminus R)^{(1)}$ for all $i$. For a given knot $J$, we write $R_\alpha(J)$ for the knot $R(\eta_1,\ldots,\eta_m;J)$, the knot obtained by infecting $R$ by $J$ along $\eta_\ell$  $(1\le \ell\le m)$ as defined in Section~\ref{sec:construction}, and we say $R_\alpha$ is a {\it doubling operator}. Then, an {\it iterated doubling operator} (at level $n$) is obtained by applying doubling operators $n$ times: $R^n_{\alpha_n}\circ R^{n-1}_{\alpha_{n-1}}\circ \cdots \circ R^1_{\alpha_1}$. In this paper, by a knot constructed using iterated doubling operators we mean a knot $(R^n_{\alpha_n}\circ R^{n-1}_{\alpha_{n-1}}\circ \cdots \circ R^1_{\alpha_1})(J)$ for some knot $J$ with vanishing Arf invariant. Here, we need $J$ to have vanishing Arf invariant to make the resulting knot $n$-solvable (see \cite[Proposition~3.1]{Cochran-Orr-Teichner:2002-1}). Furthermore, in this paper to simplify notations we consider only the iterated doubling operators where each $\alpha_i$ for $R^i$ is a single curve. Our proof can be easily and straightforwardly adapted to the general case of iterated doubling operators with multiple curves at each level.

Now since $L$ is concordant to $J=\#a_iJ_i$, there exists a homology cobordism $V$ between $M(L)$ and $M(J)$. Let $C$ be the standard cobordism between $M(J)$ and $\coprod a_i M(J_i)$. Then $V\cup C$, the union along the common boundary $M(J)$, has boundary $\partial (V\cup C) = \coprod (-a_i M(J_i)) \coprod M(L)$. 

Let us write $J_i=(R^n_{\alpha_n}\circ R^{n-1}_{\alpha_{n-1}}\circ \cdots \circ R^1_{\alpha_1})(J_i^0)$ for some knot $J_i^0$. Let $\mu_0$ be the meridian of $J_i^0$. Let $W_i$ be an $n$-solution for $M(J_i)$ such that $H_2(W_i)\cong H_2(V_i)$ for a 0-solution $V_i$ for $J_i^0$. We may assume $V_i\subset W_i$ and $\pi_1V_i\cong \Z$, which is generated by the meridian $\mu_0$ of $J_i^0$. The existence of such $W_i$ is well-known in the literature (see the proofs of \cite[Theorem~6.2]{Cochran-Harvey-Leidy:2009-2} and \cite[Proposition~4.4]{Cha:2010-1}). 

Let $W$ be the union of $V\cup C$ and $\coprod a_iW_i$ along their common boundary $\coprod a_iM(J_i)$. Then $\partial W=M(L)$. We will show that $W$ is an $(n.5, p)$-solution for $L$ for all prime $p$.

Using Mayer-Vietoris sequences, one can see that $H_2(W)\cong \oplus H_2(W_i)^{|a_i|}$, and therefore $H_2(W)\cong \oplus H_2(V_i)^{|a_i|}$. Since $\alpha_j\in \pi_1(S^3\setminus R^j)^{(1)}$ for $j=1,2,\ldots, n$, one can easily see that $\mu_0\in \pi_1M(J_i)^{(n)}$, and hence $\mu_0\in \pi_1W^{(n)}$.  Since $\pi_1V_i$ is generated by the meridian $\mu_0$, the 0-Lagrangians and 0-duals of $V_i$ for all $i$ form an $n$-Lagrangian, say $[L_j]$, and $n$-dual, say $[D_j]$, of $W$. Since $\pi^{(n)}\subset \pi^{(n)}_{cot,p}$, $[L_j]$ and $[D_j]$ are an $(n,p)$-Lagrangian and its $(n,p)$-dual, respectively, for all prime $p$.

We will show that $[L_j]$ form an $(n+1,p)$-Lagrangian of $W$ for all prime $p$, which implies that $W$ is an $(n.5, p)$-solution for all prime $p$. Since each of $L_j$ is a surface in $V_i$ for some $i$ and $V_i\subset  W$ for all $i$, it suffices to show that $\pi_1V_i$ maps to $\pi_1W^{(n+1)}_{cot,p}$ by the inclusion-induced homomorphism. Since $\pi_1V_i\cong \Z$ is generated by the meridian $\mu_0$, we only need to show $\mu_0\in \pi_1W^{(n+1)}_{cot,p}$.

Fix a prime $p$ and let $G:=\pi_1W$. Since $\mu_0\in G^{(n)}$, by the definition of $G^{(n+1)}_{cot,p}$ we only need to show $\mu_0\in  G^{(n)}_r/[G^{(n)}_r,G^{(n)}_r]$ is $\KK[t^{\pm 1}]$-torsion where $\KK$ is the skew quotient field of the group ring $\Z_p[G^{(1)}/G^{(n)}_r]$. From the iterated doubling operator construction of $J_i$, the meridian $\mu_0$ is identified with the curve $\alpha_1$. We show that $\alpha_1$ is $\KK[t^{\pm 1}]$-torsion following the arguments in the proof of \cite[Proposition~9.2]{Cochran-Harvey-Leidy:2009-2}. Let $J_i^1:=R^1_{\alpha_1}(J_i^0)$ and let $\mu_1$ be the meridian of $J_i^1$. Again since $\alpha_j\in \pi_1(S^3\setminus R^j)^{(1)}$ for all $j$, one can easily prove $\mu_1\in G^{(n-1)}$. Therefore $\pi_1M(J_i^1)\subset G^{(n-1)}$, and hence $\pi_1M(J_i^1)^{(1)}\subset G^{(n)}_r$ and $\pi_1M(J_i^1)^{(2)}\subset [G^{(n)}_r, G^{(n)}_r]$. Now let $\Delta(t)$ be the Alexander polynomial of $J_i^1$, which is the same as that of $R^1$. Since $\alpha_1\in \pi_1(S^3\setminus R^1)^{(1)}$, the polynomial $\Delta(t)$ annihilates $\alpha_1$ in the module $\pi_1M(J_i^1)^{(1)}/\pi_1M(J_i^1)^{(2)}$, and therefore $\Delta(\mu_1)$ annihilates $\alpha_1$ in the module $G^{(n)}_r/[G^{(n)}_r,G^{(n)}_r]$. Since $n\ge 2$ and $\mu_1\in G^{(n-1)}$, we have $\mu_1\in G^{(1)}$ and $\Delta(\mu_1)\in \Z_p[G^{(1)}/G^{(n)}_r]\subset \KK\subset \KK[t^{\pm 1}]$. Furthermore, since $\Delta(\mu_1)$ augments to 1, i.e., $\Delta(1)=1$, $\Delta(\mu_1)\ne 0$ in $\KK[t^{\pm 1}]$. This implies that $\alpha_1$ is $\KK[t^{\pm 1}]$-torsion. 
\end{proof}

\section{Modulo $p$ Blanchfield linking forms and algebraic $n$-solutions}\label{sec:Blanchfield linking form and algebraic n-solutions}

This section is devoted to proving Theorem~\ref{thm:eta_i}. To this end, in Subsection~\ref{subsec:Blanchfield form} we prove Theorem~\ref{thm:nontrivial} which asserts the nontriviality of some homomorphisms on first homology with twisted coefficients induced from the inclusion from $M(K)$ to an $n$-cylinder one of whose boundary components is $M(K)$. Then, in Subsection~\ref{subsec:algebraic n-solution} we introduce the notion of $\Z_p$-coefficient algebraic $n$-solutions, which generalizes the notion of algebraic $n$-solutions in \cite{Cochran-Teichner:2003-1, Cochran-Kim:2004-1}, and relevant theorems. Finally, we give a proof of Theorem~\ref{thm:eta_i} at the end of Section~\ref{sec:Blanchfield linking form and algebraic n-solutions}. 

\subsection{Modulo $p$ Blanchfield linking forms}\label{subsec:Blanchfield form}
The purpose of this subsection is to prove Theorem~\ref{thm:nontrivial} below, which plays a key role in the proof of Theorem~\ref{thm:eta_i}. Theorem~\ref{thm:nontrivial} generalizes \cite[Theorem~3.8]{Cochran-Kim:2004-1} to the case of a homomorphism on first homology whose coefficient group is a localization of a group ring with $\Z_p$ coefficients. 

First, we need the following lemma.
\begin{lemma}[{\cite[Proposition 2.5]{Cochran-Orr-Teichner:1999-1} for $R=\Q$ and \cite[Lemma 5.2]{Cha:2010-1} for $R=\Z_p$}]
\label{lem:quotient field}
	Let $R=\Q$ or $\Z_p$. If $\Gamma$ is a PTFA group, then $R\Gamma$ is an Ore domain. That is, $R\Gamma$ embeds in the (skew) quotient filed $\K=R\Gamma(R\Gamma-\{0\})^{-1}$.
\end{lemma}
Let $R=\Q$ or $\Z_p$. Let $\Gamma$ be a PTFA group such that $H_1(\Gamma)\cong \Gamma/[\Gamma,\Gamma]\cong\Z$. Let $\K$ be the (skew) quotient field of $R\Gamma$ obtained by Lemma~\ref{lem:quotient field}. Since a subgroup of a PTFA group is also PTFA, the group $[\Gamma,\Gamma]$ is PTFA. By Lemma~\ref{lem:quotient field}, the group ring $R[\Gamma, \Gamma]$ embeds into the (skew) quotient field, say $\KK$. Since $H_1(\Gamma)\cong\Z=\langle t\rangle$, we have a (noncommutative) PID $\KK[t^{\pm 1}]$ such that $R\Gamma\subset \KK[t^{\pm 1}]\subset \K$. 

\begin{theorem}[{\cite[Theorem 3.8]{Cochran-Kim:2004-1} for $R=\Q$}]
\label{thm:nontrivial}
	Let $n\ge 1$ be an integer. Let $R=\Q$ or $\Z_p$. For a knot $K$, let $W$ be an $n$-cylinder with $M(K)$ as one of its boundary components. Let $\Gamma$ be an PTFA group such that $H_1(\Gamma)\cong \Z=\langle t\rangle$ and $\Gamma^{(n)}=\{e\}$. Let $\phi\colon \pi_1W\to \Gamma$ be a homomorphism which induces an isomorphism $H_1(W)\to H_1(\Gamma)$. Let $d:=\rank_R H_1(M_\infty;R)$ where $M_\infty$ is the infinite cyclic cover of $M(K)$. Then we have
	\[
	\rank_\KK \Im\{i_*\colon H_1(M(K);\KK[t^{\pm 1}]) \to H_1(W;\KK[t^{\pm 1}])\} \ge \left\{
	\begin{array}{cl}
	(d-2)/2	&	\mbox{if } n>1,\\
	d/2	&	\mbox{if } n=1,
	\end{array}\right. 
	\]
where $i_*$ is the inclusion-induced homomorphism.
\end{theorem}
We give a proof of Theorem~\ref{thm:nontrivial} at the end of this subsection, after showing needed materials. 

 We note that the image of $i_*$ is nontrivial if $d>2$ when $n>1$ and if $d>0$ when $n=1$. In the above setting, when $R=\Q$, the rank $d$ is equal to the degree of $\Delta_K(t)$, the Alexander polynomial of $K$. When $R=\Z_p$, the rank $d$ is still equal to the degree of $\Delta_K(t)$ if the prime $p$ is bigger than the top coefficient of $\Delta_K(t)$.

To prove Theorem~\ref{thm:nontrivial}, we need to generalize the various results which were used for the proof of \cite[Theorem~3.8]{Cochran-Kim:2004-1} to the case of homology with twisted coefficients which are obtained as a localization of a group ring with $\Z_p$ coefficients. A key ingredient of the proof of Theorem~\ref{thm:nontrivial} is higher-order Blanchfield linking forms which are adapted to homology with such coefficients. 

We briefly review higher-order Blanchfield linking forms. The Blanchfield linking form in a noncommutative setting was first defined by Duval \cite{Duval:1986-1} on boundary links over a group ring of a free group. Then, Cochran--Orr--Teichner introduced the noncommutative (higher-order) Blanchfield linking form for a knot over a group ring $\Z\Gamma$ of a PTFA group $\Gamma$ and its locallizations \cite{Cochran-Orr-Teichner:1999-1}. This was generalized to the Blanchfield linking form for a knot over a group ring $\Z_p\Gamma$ for a PTFA group $\Gamma$ by Cha \cite{Cha:2010-1}. In this paper, we need the Blanchfield linking form for a knot over a (PID) localization of a group ring $\Z_p \Gamma$ for a PTFA group $\Gamma$. It will be defined in Theorem~\ref{thm:Blanchfield}, and will be used to prove Theorem~\ref{thm:nontrivial}.
 
Let $\Gamma$ be a PTFA group and $R=\Q$ or $\Z_p$. Then by Lemma~\ref{lem:quotient field} $R\Gamma$ has the (skew) quotient field of $\K$. If $\cR$ is a ring such that $R\Gamma\subset \cR\subset \K$ and $M$ is a closed 3-manifold such that $H_1(M)\cong \Z$ and $\pi_1M\to \Gamma$ is a nontrivial representation, then we have the composition of maps
\[
H_1(M;\cR)\to \overline{H^2(M;\cR)}\to \overline{H^1(M;\K/\cR)}\to \overline{\Hom_R(H_1(M;\cR), \K/\cR)}. 
\]
where the maps are Poincar\'{e} duality, the inverse of a Bockstein homomorphism, and the Kronecker evaluation map. Here, the inverse of a Bockstein homomorphism exists since $H^1(M;\K)=H^2(M;\K)=0$ and hence the Bockstein homomorphism is in fact an isomorphism (see \cite[Proposition~2.11]{Cochran-Orr-Teichner:1999-1}). Also, $\overline{H^*(M;\cR)}$ are made into right $\cR$-modules using the involution of $\cR$. 

The following theorem shows that under this setting, if $\cR$ is a PID, then the composition gives rise to a nonsingular symmetric linking form.

\begin{theorem}[{\cite[Theorem 2.13]{Cochran-Orr-Teichner:1999-1} for $R=\Q$  and essentially due to \cite[Section 5]{Cha:2010-1} for $R=\Z_p$}]
\label{thm:Blanchfield}
	Let $R=\Q$ or $\Z_p$. Let $M$ be a closed 3-manifold with $H_1(M)\cong \Z$. Let $\phi\colon \pi_1M\to \Gamma$ be a nontrivial PTFA coefficient system. Let $\cR$ be a (noncommutative) PID such that $R\Gamma\subset \cR\subset \K$ where $\K$ is the (skew) quotient field of $R\Gamma$. Then there exists a nonsingular symmetric linking form (called the {\it Blanchfield linking form}) 
\[
B\ell\colon H_1(M;\cR)\to \overline{H^2(M;\cR)}\to \overline{H^1(M;\K/\cR)}\to \overline{\Hom_R(H_1(M;\cR), \K/\cR)}. 
\]
\end{theorem}

\begin{proof}
	When $R=\Z_p$, the proof is identical to that of \cite[Theorem 2.13]{Cochran-Orr-Teichner:1999-1} (for the case $R=\Q)$ except for the following change: in the proof, replace \cite[Proposition 2.11]{Cochran-Orr-Teichner:1999-1} by \cite[Lemma 5.3]{Cha:2010-1}.
\end{proof}

Lemma~\ref{lem:exact sequence} below is needed to prove Proposition~\ref{prop:self-annihilating}. When $R=\Q$ it was proved in \cite{Cochran-Kim:2004-1}. The proof for the case $R=\Z_p$ is essentially the same as that for the case $R=\Q$, and hence it is omitted. 

\begin{lemma}[{\cite[Lemma 3.5]{Cochran-Kim:2004-1} for $R=\Q$}]
\label{lem:exact sequence}
	Let $R=\Q$ or $\Z_p$. Let $W$ be an $R$-coefficient $n$-cylinder with $M$ as one of its boundary components. Let $\Gamma$ be a PTFA group such that $\Gamma^{(n)}=\{e\}$, and let $\phi\colon \pi_1M\to \Gamma$ be a nontrivial coefficient system which extends to $\pi_1W$. Let $\cR$ be a (noncommutative) PID such that $R\Gamma\subset \cR\subset \K$ where $\K$ is the (skew) quotient field of $R\Gamma$. Then the sequence of maps 
	\[
	TH_2(W,M;\cR)\to H_1(M;\cR)\to H_1(W;\cR)
	\]
	is exact. (Here for an $\cR$-module $N$, $TN$ denotes the $\cR$-torsion submodule of $N$.)
\end{lemma}

For an $\cR$-submodule $P$ of $H_1(M;\cR)$, we define 
\[
P^\perp := \{x\in H_1(M;\cR)\,\mid\, B\ell(x)(y)=0 \mbox{ for all } y\in P\}. 
\]
We say that a submodule $P$ of $H_1(M;\cR)$ is {\it self-annihilating with respect to $B\ell$} if $P=P^\perp$. The following proposition generalizes \cite[Theorem~4.4]{Cochran-Orr-Teichner:1999-1} and \cite[Proposition~3.6]{Cochran-Kim:2004-1}.
\begin{proposition}[{\cite[Proposition 3.6]{Cochran-Kim:2004-1} for $R=\Q$}]
\label{prop:self-annihilating}
	Suppose the same hypotheses as in Lemma~\ref{lem:exact sequence}. If $P=\Ker\{i_*\colon H_1(M;\cR)\to H_1(W;\cR)\}$ where $i_*$ is the inclusion-induced homomorphism, then $P\subset P^\perp$. Moreover, if $W$ is an $R$-coefficient $n$-solution for $M$, then $P=P^\perp$. 
\end{proposition}

\begin{proof}
For $R=\Z_p$, make the following changes in the proof of \cite[Proposition 3.6]{Cochran-Kim:2004-1}: replace Corollary~3.3 and Lemma~3.5 in \cite{Cochran-Kim:2004-1} by \cite[Lemma 5.3]{Cha:2010-1} and Lemma~\ref{lem:exact sequence} above, respectively.
\end{proof}

The proof of the following lemma is based on the ideas in \cite[Sections 3 and 4]{Cochran:2002-1}.
\begin{lemma}\label{lem:rank}
	\begin{enumerate}
		\item Let $G$ be a PTFA group and $R=\Z_p$. Let $A$ be a wedge of $m$ circles. Suppose we have a nontrivial coefficient system $\pi_1A\to G$. Then $\rank_{RG}H_1( A;RG)=m-1$.
		
		\item Suppose the same hypotheses as in Theorem~\ref{thm:nontrivial}. Suppose $n>1$ and $R=\Z_p$. Then, $\rank_\KK H_1(E(K);\KK[t^{\pm 1}])\ge d-1$.
	\end{enumerate}
	
\end{lemma}
\begin{proof}
	Note that $A$ is a finite 1-complex with the Euler characteristic $1-m$. By the invariance of the Euler characteristic, we have $\rank_{RG} H_0(A;RG)-\rank_{RG}H_1(A;RG)=1-m$. Since the homomorphism $\pi_1A\to G$ is nontrivial, for the (skew) quotient field $\KK$ of $RG$ and the map $\phi\colon \pi_1A\to G\to \KK$ we have 
	\[
	H_0(A;\KK)=\KK/\{\phi(g)a-a\,\mid\,g\in \pi_1A, a\in \KK\}=0.
	\]
	Now we have $\rank_{RG} H_0(A;RG)=0$ and $\rank_{RG}H_1(A;RG)=m-1$. This proves (1).

	Let $\tilde{E}$ be the infinite cyclic cover of $E(K)$ and let $G:=[\Gamma,\Gamma]$ (therefore $\KK$ is the quotient field of $R[\Gamma,\Gamma]$). Then $H_1(E(K);\KK[t^{\pm 1}]) \cong H_1(\tilde{E};\KK)$ and $d=\rank_RH_1(\tilde{E};R)$. Since $\KK$ is the (skew) quotient field of $RG$, we have $\rank_\KK H_1(\tilde{E};\KK) = \rank_{RG} H_1(\tilde{E};RG)$, and it suffices to show that $\rank_{RG} H_1(\tilde{E};RG)\ge d-1$.
	
	Let $A$ be a wedge of $d$ circles and let $j\colon A\to \tilde{E}$ be a map which induces a monomorphism on $H_1(A;R)$. For convenience we identify $A$ with $j(A)$. Since $\tilde{E}$ has the homotopy type of a 2-complex, we may assume $(\tilde{E}, A)$ is a 2-complex. Let $C_2\to C_1\to C_0$ be the free $RG$-chain complex obtained from the cell structure of the $G$-cover of $(\tilde{E},A)$.
	
	Since $H_2(\tilde{E};R)=0$ and $H_1(A;R)\to H_1(\tilde{E};R)$ is injective, we have $H_2(\tilde{E},A;R)=0$. Therefore $C_2\otimes_{RG}R\to C_1\otimes_{RG}R$ is injective. Since PTFA groups lie in Strebel's class $D(R)$ \cite{Strebel:1974-1}, the group $G$ is in $D(R)$. Therefore $C_2\to C_1$ is also injective, and hence $H_2(\tilde{E},A;RG)=0$. Then, it follows that the map $H_1(A;RG)\to H_1(\tilde{E};RG)$ is injective. Now $\rank_{RG}H_1(\tilde{E};RG)\ge \rank_{RG}H_1(A;RG)=d-1$ where the equality holds by Part (1).	
\end{proof}

Now we are ready to give a proof of Theorem~\ref{thm:nontrivial}.

\begin{proof}[Proof of Theorem~\ref{thm:nontrivial}]
	For $R=\Z_p$, the proof is the same as the proof of \cite[Theorem 3.8]{Cochran-Kim:2004-1} (for the case $R=\Q$) except for the following changes: noticing that an $n$-cylinder is a $\Z_p$-coefficient $n$-cylinder by Proposition~\ref{prop:implication of cylinder}, replace \cite[Theorem 2.13]{Cochran-Orr-Teichner:1999-1}, \cite[Proposition 3.6]{Cochran-Kim:2004-1}, and the coefficients $\Q$ by Theorem~\ref{thm:Blanchfield}, Proposition~\ref{prop:self-annihilating}, and the coefficients $\Z_p$, respectively. One also needs to replace \cite[Corollary 4.8]{Cochran:2002-1} by Lemma~\ref{lem:rank}~(2).
\end{proof}

\subsection{Algebraic $n$-solutions}\label{subsec:algebraic n-solution}
In this section, we introduce the notion of {\it $R$-algebraic $n$-solution} where $R=\Q$ or $\Z_p$, and using it we prove Theorem~\ref{thm:eta_i}. The notion of an algebraic $n$-solutoin was introduced by Cochran and Teichner in \cite{Cochran-Teichner:2003-1} and later generalized by Cochran and the author in \cite[Section 6]{Cochran-Kim:2004-1}. The new notion of an $R$-algebraic $n$-solution generalizes an algebraic $n$-solution to the case of $\Z_p$ coefficients. An $R$-algebraic $n$-solution may be considered as an algebraic abstraction of an $n$-cylinder. 

All the results in this section are based on the ideas and results in \cite{Cochran-Teichner:2003-1} and \cite{Cochran-Kim:2004-1} which deal with the case $R=\Q$. In fact, the results in this section are mainly a $\Z_p$-coefficient version of the corresponding result with $\Q$ coefficients in \cite[Section 6]{Cochran-Kim:2004-1}. But since Section~6 in \cite{Cochran-Kim:2004-1} is quite technical, we rewrite the theorems with $\Z_p$ coefficients carefully and completely, and clarify how one should modify the proofs of the corresponding theorems in \cite{Cochran-Kim:2004-1}. 

For a group $G$, let $G_k:=G/G^{(k)}_r$ where $G^{(k)}_r$ is the $k$-th rational derived group of $G$. Since $G_k$ is a PTFA group, $\Z G_k$ (and $\Q G_k$) embeds into the (skew) quotient field which we denote by $\KK(G_k)$. By Lemma~\ref{lem:quotient field}, the group ring $\Z_pG_k$ also embeds into the (skew) quotient field, and we also denote it by $\KK(G_k)$, or $\KK_p(G_K)$ to emphasize the coefficients $\Z_p$. We write $\KK$ for $\KK(G_k)$ when it is understood from the context. 

The following is a generalization of an algebraic $n$-solution defined in \cite[Definition 6.1]{Cochran-Teichner:2003-1} and \cite[Definition 6.1]{Cochran-Kim:2004-1}.

\begin{definition}\label{def:algebraic n-solution}(\cite[Definition 6.1]{Cochran-Kim:2004-1} for $R=\Q$) Let $R=\Q$ or $\Z_p$ and let $\KK(G_k)$ be the (skew) quotient field of $RG_k$. Let $S$ be a group with $H_1(S;R)\ne 0$. Let $F$ be a free group of rank $2g$ and let $i\colon F\to S$ be a homomorphism. A nontrivial homomorphism $r\colon S\to G$ is an {\it $R$-algebraic $n$-solution} for $i\colon F\to S$ if the following hold:
	\begin{enumerate}
		\item For each $0\le k\le n-1$, the map $r_*\colon H_1(S;\KK(G_k))\to H_1(G;\KK(G_k))$ is nontrivial.
		\item For each $0\le k\le n-1$, the map $i_*\colon H_1(F;\KK(G_k))\to H_1(S;\KK(G_k))$ is surjective. 
	\end{enumerate}
	A $\Q$-algebraic $n$-solution is also called an {\it algebraic $n$-solution}. 
\end{definition}

\begin{remark}\label{rem:algebraic n-solution}
	\begin{enumerate}	
		\item For $k<n$, an $R$-algebraic $n$-solution is an $R$-algebraic $k$-solution. 
		\item Since $\KK(G_k)$ is a flat $RG_k$-module \cite[Proposition II.3.5]{Stenstrom:1975}, we have $H_1(-;\KK(G_k))\cong H_1(-;RG_k)\otimes_{RG_k}\KK(G_k)$.
		\item We have $H_1(G;RG_k)\cong H_1(G;\Z G_k)\otimes_\Z R\cong G^{(k)}_r/[G^{(k)}_r,G^{(k)}_r]\otimes_\Z R$.
		\item The case $k=0$ in the condition (2) implies that $H_1(F;R)\to H_1(S;R)$ is surjective.
	\end{enumerate}
\end{remark}

The following proposition shows the relationship between $n$-cylinders and $R$-algebraic $n$-solutions. Roughly speaking, for a knot $K$ and an $n$-cylinder $W$ for $M(K)$, the homomorphism on the fundamental groups of the infinite cyclic covers of $M(K)$ and $W$ induced from the inclusion gives rise to an $R$-algebraic $n$-solution. 

\begin{proposition}[{\cite[Proposition 6.3]{Cochran-Kim:2004-1} for $R=\Q$}]
\label{prop:n-cylinder and n-solution}
	Let $n\ge 1$. Let $K$ be a knot with nontrivial Alexander polynomial $\Delta_K(t)$. Suppose the degree of $\Delta_K(t)$ is greater than 2 if $n>1$. Let $W$ be an $n$-cylinder with $M(K)$ as one of its boundary components. Let $\Sigma$ be a capped-off Seifert surface for $K$. Let $S:=\pi_1M(K)^{(1)}$ and $G:=\pi_1W^{(1)}$. Let $R:=\Q$ or $\Z_p$ where $p$ is a prime greater than the top coefficient of $\Delta_K(t)$. Let $F$ be a free group of rank $2g$ where $g$ is the genus of $\Sigma$ and let $F\to \pi_1(M(K)\setminus \Sigma)$ be a homomorphism inducing an isomorphism on $H_1(F;R)$. Let $i$ be the composition $F\to \pi_1(M(K)\setminus \Sigma)\to S$. Then, the inclusion-induced map $j\colon S\to G$ is an $R$-algebraic $n$-solution for $i\colon F\to S$. 
\end{proposition}
\begin{proof}
	For $R=\Z_p$, modify the proof of Proposition~6.3 in \cite{Cochran-Kim:2004-1} (for the case $R=\Q$) as follows: 
	\begin{enumerate}
		\item Replace the coefficients $\Z$ by $\Z_p$ and let $\KK$ be the (skew) quotient field of $\Z_p\Gamma^{(1)}=\Z_pG_k$.
		\item Replace \cite[Theorem 3.8]{Cochran-Kim:2004-1} by Theorem~3.5 in this paper. 
		\item In the proof of \cite[Proposition~6.3]{Cochran-Kim:2004-1}, it was given that $d:=\rank_\Q H_1(M_\infty;\Q)$ where $M_\infty$ is the infinite cyclic cover of $M(K)$, and it was used that $d$ is equal to the degree of $\Delta_K(t)$. In our case with $\Z_p$ coefficients, we set $d:=\rank_{\Z_p} H_1(M_\infty;\Z_p)$. Note that $d$ is still equal to the degree of $\Delta_K(t)$ since $p$ is greater than the top coefficient of $\Delta_K(t)$ by our hypothesis. 
		\item To establish Property (2) in Definition~\ref{def:algebraic n-solution}, choose a map $W\to Y=M(K)\setminus \Sigma$ inducing $F\to \pi_1Y$ that is 1-connected on homology with $\Z_p$ coefficients.
		\item Use the $\Z_p$-coefficient version of \cite[Proposition 2.10]{Cochran-Orr-Teichner:1999-1} to show that the map $F\to \pi_1Y$ induces a 1-connected map on homology with $\KK(G_k)$ coefficients: Proposition~2.10 in \cite{Cochran-Orr-Teichner:1999-1} holds with $\Z_p$ coefficients since $G_k$ lies in Strebel's class $D(\Z_p)$ (see Lemma~\ref{lem:amenable and D(R)}).
		\item The proof of \cite[Proposition~6.3]{Cochran-Kim:2004-1} uses Harvey's work in \cite{Harvey:2005-1} with the (skew) quotient field $\KK$ of $\Z G_k$. We can still use Harvey's work in the same way for the case of $\Z_p$ coefficients with our $\KK=\KK_p(G_k)$, the (skew) quotient field of $\Z_p G_k$.
\end{enumerate}
\end{proof}
\noindent We note that in the proofs of Proposition~\ref{prop:n-cylinder and n-solution} and \cite[Proposition 6.3]{Cochran-Kim:2004-1}, it is implicitly proved that the map $i_*\colon H_1(F;\KK(G_k))\to H_1(S;\KK(G_k))$ is surjective for all $k\ge 0$.

By Remark~\ref{rem:algebraic n-solution}~(3), for a group $S$ and $R=\Q$ or $\Z_p$, an element of $S^{(n)}$ can be considered as an element of $H_1(S;RS_n)$. 

\begin{theorem}[{\cite[Theorem 6.4]{Cochran-Kim:2004-1} for $R=\Q$ only}]
\label{thm:special tuple}
 Let $F$ be a free group of rank $2g$ and let $i\colon F\to S$ be a group homomorphism. For each $n\ge 0$, there exists a finite collection $\cP_n$ of sets of $2g-1$ (if $n>0$) or $2g$ (if $n=0$) elements of $F^{(n)}$ which satisfies the following:
	
	If $r\colon S\to G$ is an $R$-algebraic $n$-solution for $i\colon F\to S$ for both $R=\Q$ and $R=\Z_p$, then there is an element in $\cP_n$ which maps to a generating set of $H_1(S;\KK(G_n))$ under the composition 
	\[
	F^{(n)}\to S^{(n)}\to H_1(S;RS_n)\xrightarrow{r_*}H_1(S;RG_n)\to H_1(S;\KK(G_n)).
	\]
	for both $R=\Q$ and $R=\Z_p$.  
\end{theorem}

An element of $\cP_n$ is called an (unordered) {\it tuple}, and a tuple mapped to a generating set of $H_1(S;\KK(G_n))$ in the above theorem is called a {\it special tuple}. In the proof of Theorem~\ref{thm:eta_i}, the $\eta_i$ will be defined as the image of the elements of the tuples in $\cP_{n-1}$ under the homomorphism $i\colon F\to S$, which is a finite set, and the existence of a special tuple will give us the existence of the desired curve $\eta_i$. 

\begin{proof}[Proof of Theorem~\ref{thm:special tuple}]
	Suppose $x_1,x_2,\ldots, x_{2g}$ generate $F$. Let $\cP_0=\{\{x_1,x_2,\ldots, x_{2g}\}\}$. When $n=0$, the above composition becomes $F\to H_1(S;R)$ and the theorem follows since $H_1(F;R)\to H_1(S;R)$ is surjective by  Remark~\ref{rem:algebraic n-solution}~(4) and $x_i$ generate $H_1(F;R)\cong R^{2g}$.
	
	Now assume $n\ge 1$. We define $\cP_n$ inductively as in the proof of \cite[Theorem 6.4]{Cochran-Kim:2004-1}. Define 
	\[\cP_1:=\{\{[x_i,x_1],\dots,
	[x_i,x_{i-1}],[x_i,x_{i+1}],\ldots,[x_i,x_{2g}]\}\,\,|\,\,1\le i
	\le 2g\},
	\] 
	which is a set of $(2g-1)$-tuples. Suppose $\cP_k$ has been constructed for $k\ge 1$. We define $\cP_{k+1}$ as follows: a $(2g-1)$-tuple $\{z_1,...,z_{2g-1}\}$ is in $\cP_{k+1}$ if and only if there is $\{w_1,...,w_{2g-1}\}\in \cP_k$ such that for each $1\le i\le 2g-1$, $z_i=[w_i,w_i^{x_j}]$ for some $j$ with $1\le j\ne i\le 2g$ or $z_i=[w_i,w_k]$ for some $k$ with $1\le k\ne i\le 2g-1$. Here $w_i^{x_j}$ denotes $x_j^{-1}w_ix_j$.
	
	By \cite[Theorem 6.4]{Cochran-Kim:2004-1} the conclusion of the theorem holds when $R=\Q$. That is, there is an element in $\cP_n$ mapping to a generating set of $H_1(S;\KK(G_n))$ where $\KK(G_n)$ is the (skew) quotient field of $\Q G_n$.
	
	We assert that the above element in $\cP_n$ also maps to a generating set of $H_1(S;\KK(G_n))$ when $\KK(G_n)=\KK_p(G_n)$, the (skew) quotient field of $\Z_pG_n$. This is proved by modifying the proof of Theorem~6.4 in \cite{Cochran-Kim:2004-1} as follows: 
	\begin{enumerate}
		\item Replace the group rings over $\Z$ coefficients by the corresponding group rings over $\Z_p$ coefficients. For example, change $\Z F_n$ to $\Z_p F_n$. 
		\item Let $\KK(G_n)$ denote the (skew) quotient field of $\Z_pG_n$.  
		\item Use Lemma~\ref{lem:rank}~(1) in this paper instead of \cite[Lemma 3.9]{Cochran:2002-1}.
		\item Use Lemma~\ref{lem:good tuple} below instead of \cite[Lemma 6.5]{Cochran-Kim:2004-1}. Note that we need the hypothesis that $r\colon S \to G$ is a ($\Q$-)algebraic $n$-solution in Lemma~\ref{lem:good tuple}. 
		\item In the last part of the proof, one needs to use that $H_1(F;\Z_pG_n)\to H_1(S:\Z_pG_n)$ is surjective after tensoring with $\KK_p(G_n)$. Note that this is where we use the assumption that $r\colon S\to G$ is a $\Z_p$-algebraic $n$-solution.
	\end{enumerate}
\end{proof} 

Let $R$ be a ring. For each $1\le i\le 2g$, let $\partial_i;F\to \Z F\to RF$ be the Fox free differential calculus defined by $\partial_i (x_j)=\delta_{ij}$ and $\partial _i(gh) = \partial_i g + (\partial_i h)g^{-1}$. For $k\ge 0$, denote by $\pi_k$ the quotient map $RF\to RF_k$. By abuse of notation, denote by $r$ the map $RF_k\to RG_k$ induced from $r\circ i\colon F\to S\to G$. We denote by $d_i^k$ the composition $r\circ\pi_k\circ \partial_i\colon F\to RG_k$ for $1\le i\le 2g$. We simply write $d_i$ for $d_i^k$ when the superscript is understood from the context.
 
\begin{lemma}[{\cite[Lemma 6.5]{Cochran-Kim:2004-1}) for $R=\Q$}]
\label{lem:good tuple}
 Let $n\ge 1$. Suppose we are given an algebraic $n$-solution $r\colon S\to G$ for $i\colon F\to S$. Then, there exists a $(2g-1)$-tuple $\{w_1,\ldots,w_{2g-1}\}\in \cP_n$ such that the $2g-1$ vectors $\{(d_1(w_i),d_2(w_i),\ldots, d_{2g-1}(w_i))\,\mid\, 1\le i\le 2g-1\}$ in $(RG_n)^{2g-1}$ are right linearly independent over $RG_n$ for $R=\Q$ and $R=\Z_p$ for all prime $p$.
\end{lemma}
\begin{proof}
	In the proof of Lemma ~6.5 in \cite{Cochran-Kim:2004-1}, for each $1\le k\le n$, a tuple $\{w_1^k,\ldots, w_{2g-1}^k\}\in \cP_k$ was constructed inductively such that  $\{(d_1(w_i^k),d_2(w_i^k),\ldots, d_{2g-1}(w_i^k))\,\mid\, 1\le i\le 2g-1\}$ in $(\Q G_k)^{2g-1}$ are right linearly independent over $\Q G_k$ using the fact that $r\colon S\to G$ is an algebraic $n$-solution. 
	
	We will show that for each $k$, the vectors $\{(d_1(w_i^k),d_2(w_i^k),\ldots, d_{2g-1}(w_i^k))\,\mid\, 1\le i\le 2g-1\}$ are also right linearly independent over $\Z_pG_k$, and it will complete the proof by taking $\{w_1,\ldots,w_{2g-1}\}=\{w_1^n,\ldots,w_{2g-1}^n\}$. The proof follows the lines of the proof of \cite[Lemma 6.5]{Cochran-Kim:2004-1}, and we explain the needed modifications below. 
	
	For $k=1$, in the proof of \cite[Lemma 6.5]{Cochran-Kim:2004-1} we assume that $(r\circ \pi_1 )(x_{2g})$ is nontrivial in $G_1$ and take $\{w_1^1,\ldots, w_{2g-1}^1\}=\{[x_{2g},x_1],\ldots, [x_{2g},x_{2g-1}]\}$. Then, it was shown that the vectors $\{(d_1(w_i^1),d_2(w_i^1),\ldots, d_{2g-1}(w_i^1))\,\mid\, 1\le i\le 2g-1\}$ are linearly independent over $\Q G_1$ (in the first paragraph of \cite[p.1428]{Cochran-Kim:2004-1}). Using the same argument, it can be seen that $\{w_1^1,\ldots, w_{2g-1}^1\}$ are also linearly independent over $\Z_pG_1$.
	
	We use an induction argument. Suppose that for $k<n$ it has been shown that the vectors $\{(d_1(w_i^k),d_2(w_i^k),\ldots, d_{2g-1}(w_i^k))\,\mid\, 1\le i\le 2g-1\}$ are linearly independent over $\Z_pG_k$. In \cite{Cochran-Kim:2004-1} it was shown that we may assume $(r\circ \pi_{k+1})(w_1^k)\ne e$ in $G_{k+1}$. Then $\{w_1^{k+1},\ldots, w_{2g-1}^{k+1}\}$ was defined in \cite{Cochran-Kim:2004-1} as follows: we take $w_i^{k+1}=[w_i^k,w_i^{x_{2g}}]$ if $(r\circ \pi_{k+1})  (w_i^k)\ne e$ in $G_{k+1}$ and $w_i^{k+1}=[w_i^k, w_1^k]$, otherwise.
	
	Then it was shown in \cite{Cochran-Kim:2004-1} that in $\Z G_{k+1}$, for some $t_i\in \Z F$
	\[
	(d_1^{k+1}(w_i^{k+1}),d_2^{k+1}(w_i^{k+1}),\ldots, d_{2g-1}^{k+1}(w_i^{k+1})) = (d_1^{k+1}(w_i^k),d_2^{k+1}(w_i^k),\ldots, d_{2g-1}^{k+1}(w_i^k))\cdot (r\circ \pi_{k+1})(t_i)
	\]
	where $(r\circ \pi_{k+1})(t_i)\ne e$ in $\Z G_{k+1}$. We note that using the same argument in \cite{Cochran-Kim:2004-1} one can show that $(r\circ \pi_{k+1})(t_i)\ne e$ in $\Z_p G_{k+1}$ as well. This implies that it suffices to show that the vectors $\{(d_1^{k+1}(w_i^k),d_2^{k+1}(w_i^k),\ldots, d_{2g-1}^{k+1}(w_i^k))\,\mid\, 1\le i\le 2g-1\}$ are linearly independent over $\Z_p G_{k+1}$. 
	
	For simplicity,  for $1\le i\le 2g-1$, let $\mathbf{v_i^{k+1}}=(d_1^{k+1}(w_i^k),d_2^{k+1}(w_i^k),\ldots, d_{2g-1}^{k+1}(w_i^k))$ and $\mathbf{v_i^k}=(d_1^k(w_i^k),d_2^k(w_i^k),\ldots, d_{2g-1}^k(w_i^k))$. By the induction hypothesis, the vectors $\mathbf{v_i^k}$, $1\le i\le 2g-1$, are linearly independent over $\Z_p G_{k+1}$ (and over $\Q G_{k+1}$).
	
	Let $H=G_r^{(k)}/G_{r+1}^{(k)} = \Ker \{G_{k+1}\to G_k\}$. Then $H$ is a torsion-free abelian group, and hence lies in Strebel's class $D(\Z_p)$ \cite{Strebel:1974-1}. In the last paragraph of the poof of \cite[Lemma 6.5]{Cochran-Kim:2004-1} it was shown that linear independence of the vectors $\mathbf{v_i^k}$ over $\Z G_k$ implies linear independence of the vectors $\mathbf{v_i^{k+1}}$ over $\Z G_{k+1}$. Using the same argument and the fact that the group $H$ lies in $D(\Z_p)$, changing the coefficient group for group rings from $\Z$ to $\Z_p$ in the proof, one can show that linear independence of the vectors $\mathbf{v_i^k}$ over $\Z_p G_k$ implies linear independence of the vectors $\mathbf{v_i^{k+1}}$ over $\Z_p G_{k+1}$. This completes the proof.	
\end{proof}

Now we give a proof of Theorem~\ref{thm:eta_i}.
\begin{proof}[proof of Theorem~\ref{thm:eta_i}] Let $p$ be a prime greater than the top coefficient of $\Delta_K(t)$ and $W$  an $n$-cylinder one of whose boundary components is $M(K)$. Let $\cP=(R_0, \ldots, R_n)$ where $R_i=\Q$ for $i\le n-1$ and $R_n=\Z_p$. With this $\cP$, recall that for a group $G$ and $k\le n$, $\cP^kG=G^{(k)}_r$, the $k$-th rational derived group of $G$. Recall that for a group $G$, we denote $G/G^{(k)}_r$ by $G_k$.
	
	For convenience, we also denote by $\Sigma$ the capped-off Seifert surface for $K$. Supose $\Sigma$ has genus $g$. Let $M:=M(K)$, $S:=\pi_1M^{(1)}$, and $G:=\pi_1W^{(1)}$. Let $F$ be a free group of rank $2g$ and $F\to \pi_1(M\setminus\Sigma)$ a homomorphism inducing an isomorphism on $H_1(F;R)$ for $R=\Q$ and $\Z_p$. Let $i\colon F\to \pi_1(M\setminus\Sigma)\to S$ be a map induced from this map. Since $G_{n-1}$ is PTFA \cite[Proposition 2.1]{Harvey:2006-1},  $\Z_pG_{n-1}$ embeds into the (skew) quotient field, say $\KK$, by Lemma~\ref{lem:quotient field}. 
	
	Let $M_\infty$ and $W_\infty$ denote the infinite cyclic covers of $M$ and $W$, respectively. Let $\Gamma:=\pi_1W/\cP^n\pi_1W=\pi_1W/\pi_1W^{(n)}_r$. Then $\Gamma$ is a PTFA group such that $\Gamma^{(n)}=\{e\}$. Since  $H_1(\Gamma)\cong \Z=\langle t\rangle$ and 
	\[
	[\Gamma,\Gamma]=\pi_1(W)^{(1)}/\pi_1W^{(n)}_r=G/G^{(n-1)}_r=G_{n-1},
	\] 
	we have $\Gamma\cong G_{n-1}\rtimes \langle t\rangle$. Therefore, we have 
	\begin{align*}
H_1(W_\infty;\KK) & \cong H_1(W_\infty;\Z_pG_{n-1})\otimesover{\Z_pG_{n-1}}\KK \\
				& \cong H_1(W;\Z_p[G_{n-1}\rtimes \langle t\rangle])\otimesover{\Z_pG_{n-1}}\KK \\
				& \cong H_1(W;\KK[t^{\pm 1}]).
	\end{align*}
	Similarly, we have $H_1(M_\infty;\KK)\cong H_1(M;\KK[t^{\pm 1}])$.
	
	Now by Proposition~\ref{prop:n-cylinder and n-solution}, the inclusion-induced homomorphism $j\colon S\to G$ is an $R$-algebraic $n$-solution for $i\colon F\to S$ for $R=\Q$ and $\Z_p$, and hence an $R$-algebraic $(n-1)$-solution for both $R=\Q$ and $\Z_p$ (see Remark~\ref{rem:algebraic n-solution}~(1)). Let $\cP_{n-1}$ be the finite collection of tuples of $2g-1$ (if $n>1$) or $2g$ (if $n=1$) elements of $F^{(n-1)}$ obtained using Theorem~\ref{thm:special tuple} with the map $i\colon F\to S$ above. The image of the elements of the tuples in $\cP_{n-1}$ under the homomorphism $i\colon F\to S$ is a finite subset of $S^{(n-1)}=\pi_1M^{(n)}$, which we denote by $\eta_1, \eta_2, \ldots \eta_m$. Since the map $i$ factors through $\pi_1(M\setminus\Sigma)$,  we can find the representatives $\eta_i$ in $S^3 \setminus\Sigma$, and by crossing change we can make $\eta_i$ form an unlink. Now since $j\colon S\to G$ is an $R$-algebraic $(n-1)$-solution for $i\colon F\to S$, by Theorem~\ref{thm:special tuple} there exists a $(2g-1)$-tuple $\{w_1,w_2,\ldots, w_{2g-1}\} \subset F^{(n-1)}$ (or a $2g$-tuple $\{w_1,\ldots, w_{2g}\}\subset F$ if $n=1$) which maps to a generating set of $H_1(S;\KK)$. Therefore, since $H_1(S;\KK)\cong H_1(M_\infty;\KK)\cong H_1(M;\KK[t^{\pm 1}])$, the $\eta_i$ can be regarded as a generating set of $H_1(M;\KK[t^{\pm 1}])$ (as a $\KK$-module).  
	
	 Let $d:=\deg \Delta_K(t)=\rank_\Q H_1(M_\infty;\Q)$. Since $p$ is greater than the top coefficient of $\Delta_K(t)$, we have $d=\rank_{\Z_p}H_1(M_\infty;\Z_p)$. Note that $d\ge 4 $ if $n\ge 2$ and $d\ge 2$ if $n=1$ by assumption. Since $\Gamma$ is a PTFA group such that $\Gamma^{(n)}=\{e\}$, by Theorem~\ref{thm:nontrivial} this implies that the map $H_1(M;\KK[t^{\pm 1}])\to H_1(W;\KK[t^{\pm 1}])$ is nontrivial. Therefore, since the $\eta_i$ generate $H_1(M;\KK[t^{\pm 1}])$ as a $\KK$-module, there exists some $\eta_i$ which maps to a nontrivial element in $H_1(W;\KK[t^{\pm 1}])$.
	 
	 Note that $H_1(W;\KK[t^{\pm 1}])\cong H_1(W;\Z_p\Gamma)\otimesover{\Z_pG_{n-1}}\KK$. By the definition of $\cP^{n+1}\pi_1W$, the group $\cP^n\pi_1W/\cP^{n+1}\pi_1W$ injects to $H_1(\pi_1W;\Z_p\Gamma)\cong H_1(W;\Z_p\Gamma)$. Since $\eta_i\in \pi_1M^{(n)}$, it maps into $\pi_1W^{(n)}$, hence into $\cP^n\pi_1W$. From these observations, one can deduce that the $\eta_i$, which maps to a nontrivial element in $H_1(W;\KK[t^{\pm 1}])$, also maps nontrivially to $\cP^n\pi_1W/\cP^{n+1}\pi_1W$. In particular, we see that $\eta_i\notin \cP^{n+1}\pi_1W$.
	 
	 Finally, by Lemma~\ref{lem:axes bounding grope of height n} we can homotope all $\eta_i$ such that all of $\eta_i$ bound capped gropes of height $n$ which are disjointly embedded in $S^3\setminus K$ as desired.	 	
\end{proof}

%\bibliographystyle{amsalpha}
%\renewcommand{\MR}[1]{}
%\bibliography{research}
\providecommand{\bysame}{\leavevmode\hbox to3em{\hrulefill}\thinspace}
\providecommand{\MR}{\relax\ifhmode\unskip\space\fi MR }
\providecommand{\MRhref}[2]{%
  \href{http://www.ams.org/mathscinet-getitem?mr=#1}{#2}
}
\providecommand{\href}[2]{#2}

\end{document}